\documentclass[12pt]{amsart}
\usepackage{amsmath,amssymb,amsbsy,amsfonts,amsthm,latexsym,
                        amsopn,amstext,amsxtra,euscript,amscd,mathrsfs,color,bm,mathtools}
                        
\usepackage{float} 
\usepackage[english]{babel}
\usepackage{url}
\usepackage[colorlinks,linkcolor=blue,anchorcolor=blue,citecolor=blue,backref=page]{hyperref}

\usepackage[margin=3.1cm]{geometry}

\usepackage{color}
\usepackage[norefs,nocites]{refcheck}

\renewcommand*{\backref}[1]{}
\renewcommand*{\backrefalt}[4]{%
    \ifcase #1 (Not cited.)%
    \or        (p.\,#2)%
    \else      (pp.\,#2)%
    \fi}

\begin{document}

\newtheorem{theorem}{Theorem}
\newtheorem{lemma}[theorem]{Lemma}
\newtheorem{example}[theorem]{Example}
\newtheorem{algol}{Algorithm}
\newtheorem{corollary}[theorem]{Corollary}
\newtheorem{prop}[theorem]{Proposition}
\newtheorem{definition}[theorem]{Definition}
\newtheorem{question}[theorem]{Question}
\newtheorem{problem}[theorem]{Problem}
\newtheorem{remark}[theorem]{Remark}
\newtheorem{conjecture}[theorem]{Conjecture}

\newcommand{\llrrparen}[1]{
  \left(\mkern-3mu\left(#1\right)\mkern-3mu\right)}

\newcommand{\commM}[1]{\marginpar{%
\begin{color}{red}
\vskip-\baselineskip 
\raggedright\footnotesize
\itshape\hrule \smallskip M: #1\par\smallskip\hrule\end{color}}}

\newcommand{\commA}[1]{\marginpar{%
\begin{color}{blue}
\vskip-\baselineskip 
\raggedright\footnotesize
\itshape\hrule \smallskip A: #1\par\smallskip\hrule\end{color}}}
\def\xxx{\vskip5pt\hrule\vskip5pt}


\def\cA{{\mathcal A}}
\def\cB{{\mathcal B}}
\def\cC{{\mathcal C}}
\def\cD{{\mathcal D}}
\def\cE{{\mathcal E}}
\def\cF{{\mathcal F}}
\def\cG{{\mathcal G}}
\def\cH{{\mathcal H}}
\def\cI{{\mathcal I}}
\def\cJ{{\mathcal J}}
\def\cK{{\mathcal K}}
\def\cL{{\mathcal L}}
\def\cM{{\mathcal M}}
\def\cN{{\mathcal N}}
\def\cO{{\mathcal O}}
\def\cP{{\mathcal P}}
\def\cQ{{\mathcal Q}}
\def\cR{{\mathcal R}}
\def\cS{{\mathcal S}}
\def\cT{{\mathcal T}}
\def\cU{{\mathcal U}}
\def\cV{{\mathcal V}}
\def\cW{{\mathcal W}}
\def\cX{{\mathcal X}}
\def\cY{{\mathcal Y}}
\def\cZ{{\mathcal Z}}

\def\C{\mathbb{C}}
\def\D{\mathbb{D}}
\def\I{\mathbb{I}}
\def\F{\mathbb{F}}
\def\H{\mathbb{H}}
\def\K{\mathbb{K}}
\def\G{\mathbb{G}}
\def\Z{\mathbb{Z}}
\def\R{\mathbb{R}}
\def\Q{\mathbb{Q}}
\def\N{\mathbb{N}}
\def\M{\textsf{M}}
\def\U{\mathbb{U}}
\def\P{\mathbb{P}}
\def\A{\mathbb{A}}
\def\p{\mathfrak{p}}
\def\n{\mathfrak{n}}
\def\X{\mathcal{X}}
\def\x{\textrm{\bf x}}
\def\w{\textrm{\bf w}}
\def\z{\mathrm{z}}
\def\ovQ{\overline{\Q}}
\def\rank#1{\mathrm{rank}#1}
\def\wf{\widetilde{f}}
\def\wg{\widetilde{g}}
\def\comp{\hskip -2.5pt \circ  \hskip -2.5pt}
\def\({\left(}
\def\){\right)}
\def\[{\left[}
\def\]{\right]}
\def\<{\langle}
\def\>{\rangle}
\newcommand*\diff{\mathop{}\!\mathrm{d}}
\def\Lip{\mathrm{Lip}}

\def\BP{\mathbb{P}_{\mathrm{Berk}}^1}
\def\BA{\mathbb{A}_{\mathrm{Berk}}^1}
\def\cJBv{\mathcal{J}_{v,\mathrm{Berk}}}
\def\cFBv{\mathcal{F}_{v,\mathrm{Berk}}}

\def\gen#1{{\left\langle#1\right\rangle}}
\def\genp#1{{\left\langle#1\right\rangle}_p}
\def\genPs{{\left\langle P_1, \ldots, P_s\right\rangle}}
\def\genPsp{{\left\langle P_1, \ldots, P_s\right\rangle}_p}

\def\e{e}

\def\eq{\e_q}
\def\fh{{\mathfrak h}}

\def\lcm{{\mathrm{lcm}}\,}

\def\l({\left(}
\def\r){\right)}
\def\fl#1{\left\lfloor#1\right\rfloor}
\def\rf#1{\left\lceil#1\right\rceil}
\def\mand{\qquad\mbox{and}\qquad}

\def\jt{\tilde\jmath}
\def\ellmax{\ell_{\rm max}}
\def\llog{\log\log}

\def\diam{{\mathrm{diam}}}
\def\m{{\mathfrak{m}}}
\def\ud{{\mathrm{d}}}
\def\ch{\hat{h}}
\def\GL{{\rm GL}}
\def\Orb{\mathrm{Orb}}
\def\Per{\mathrm{Per}}
\def\larg{\mathrm{larg}}
\def\Preper{\mathrm{Preper}}
\def \S{\mathcal{S}}
\def\vec#1{\mathbf{#1}}
\def\ov#1{{\overline{#1}}}
\def\Gal{{\rm Gal}}

\newcommand{\bfalpha}{{\boldsymbol{\alpha}}}
\newcommand{\bfomega}{{\boldsymbol{\omega}}}

\newcommand{\canheight}[1]{{\hat h_{#1}}}
\newcommand{\loccanheight}[2]{{\hat h_{#2,#1}}}
\newcommand{\relsupreper}[3]{{\cP_{#1,#2}(#3)}}
\newcommand{\disunit}[3]{{\cA_{#1,#2}(#3)}}
\def\rat{f}
\def\pol{f}
\def\unipol{f_c}

\newcommand{\Ch}{{\operatorname{Ch}}}
\newcommand{\Elim}{{\operatorname{Elim}}}
\newcommand{\proj}{{\operatorname{proj}}}
\newcommand{\h}{{\operatorname{h}}}

\newcommand{\hh}{\mathrm{h}}
\newcommand{\aff}{\mathrm{aff}}
\newcommand{\Spec}{{\operatorname{Spec}}}
\newcommand{\Res}{{\operatorname{Res}}}

\numberwithin{equation}{section}
\numberwithin{theorem}{section}

\def\house#1{{%
    \setbox0=\hbox{$#1$}
    \vrule height \dimexpr\ht0+1.4pt width .5pt depth \dimexpr\dp0+.8pt\relax
    \vrule height \dimexpr\ht0+1.4pt width \dimexpr\wd0+2pt depth \dimexpr-\ht0-1pt\relax
    \llap{$#1$\kern1pt}
    \vrule height \dimexpr\ht0+1.4pt width .5pt depth \dimexpr\dp0+.8pt\relax}}

\newcommand{\Address}{{
\bigskip
\footnotesize
\textsc{Department of Pure Mathematics and Mathematical Statistics, University of Cambridge, Cambridge CB3 0WB, UK}\par\nopagebreak
\textit{E-mail address:} \texttt{mjy28@cam.ac.uk}
}}


\title[]
{Effective bounds on $S$-integral preperiodic points for polynomials}

\begin{abstract}
Given a polynomial $\pol$ defined over a number field $K$, we make effective certain special cases of a conjecture of S. Ih, on the finiteness of $\pol$-preperiodic points which are $S$-integral with respect to a fixed non-preperiodic point $\alpha$. As an application, we obtain bounds on the number of $S$-units in the doubly indexed sequence $\{ \pol^n(\alpha) - \pol^m(\alpha) \}_{n > m \geq 0}$. In the case of a unicritical polynomial $\unipol(z)=z^2+c$, with $\alpha$ fixed to be the critical point 0, for parameters $c$ outside a small region, we give an explicit bound which depends only on the number of places of bad reduction for $\unipol$. As part of the proof, we obtain novel lower bounds for the $v$-adically smallest preperiodic point of $\unipol$ for each place $v$ of $K$.
\end{abstract}

\author {Marley Young}

\maketitle

\vspace{-0.5cm}

\section{Introduction}

\subsection{Motivation and background} Let $K$ be a number field with algebraic closure $\overline K$, and let $S$ be a finite set of places of $K$ including all the archimedean ones. Given two points $\alpha, \beta \in \P^1(\overline K)$, we say that $\beta$ is \emph{$S$-integral relative to $\alpha$} if $\alpha$ and $\beta$ have distinct reduction with respect to each place of $\overline K$ lying over the places of $K$ outside $S$. A notion of $S$-integrality can more generally be defined relative to an effective divisor on an abelian variety (see for example \cite{GI}).


Problems of unlikely intersections have fueled recent development in arithmetic geometry and provided a deep connection to arithmetic dynamics. For example, S. Ih \cite[Conjecture~1.2]{BIR} has conjectured, given an abelian variety $A/K$ and an effective divisor $D$ on $A$, that the set of torsion points of $A(\overline K)$ which are $S$-integral relative to $D$ is not Zariski dense in $A$. This conjecture can be viewed as similar to the Manin-Mumford conjecture, proved by Raynaud \cite{Ra}, but integral points are considered since $A \setminus \mathrm{supp}(D)$ is non-compact. The cases of elliptic curves and the multiplicative group were proved by Baker, Ih and Rumely in \cite{BIR}. Links from arithmetic geometry to arithmetic dynamics are often made through an analogy between torsion points on abelian varieties and preperiodic points of rational maps. Indeed, in this manner, Ih further conjectured a general dynamical phenomenon.

Given a self map $\rat: X \to X$ of a set $X$, let $\rat^n$ denote the $n$-fold composition given by
$$
\rat^n = \underbrace{\rat \circ \cdots \circ \rat}_{n \text{ times}}.
$$
A point $\alpha \in X$ is \emph{preperiodic} for $\rat$ if $\rat^{n+m}(\alpha) = \rat^m(\alpha)$ for some $n \geq 1$ and $m \geq 0$. For $Y \subseteq X$, let $\mathrm{PrePer}(f,Y)$ denote the set of elements of $Y$ which are preperiodic for $f$.

\begin{conjecture}[Ih] \label{conj:Ih}
Let $K$ be a number field, and let $S$ be a finite set of places containing all the archimedean ones. If $\rat : \P^1 \to \P^1$ is a rational function of degree at least two defined over $K$, and $\alpha \in \P^1(\overline K)$ is a point which is not preperiodic, then there are at most finitely many preperiodic points $\beta \in \P^1(\overline K)$ that are $S$-integral with respect to $\alpha$.
\end{conjecture}

We say that a point $\alpha \in \P^1(\overline K)$ is \emph{totally Fatou} at a place $v$ of $K$ (with respect to the rational function $\rat$) if each of the $[K(\alpha):K]$ embeddings of $\alpha$ into $\P^1(\C_v)$ is contained in the $v$-adic Fatou set of $\rat$. Petsche \cite{P} proved Conjecture~\ref{conj:Ih} in the case where the non-preperiodic point $\alpha$ is totally Fatou at all places of $K$. In more generality, Conjecture~\ref{conj:Ih} has only been proved for very special maps, namely, those arising from endomorphisms of commutative algebraic groups for which the aforementioned \cite[Conjecture~1.2]{BIR} has been established. Baker, Ih and Rumely \cite{BIR} have proved the conjecture when $\rat$ is a Latt\'{e}s map associated to an elliptic curve or $\rat(z)=z^d$ for some $d \geq 2$. Also, in \cite{IT}, Ih and Tucker prove the conjecture for Chebyshev maps.

A key tool in the proof of each of the above cases of Conjecture~\ref{conj:Ih} is the equidistribution of points of small height. Recently, we have seen the effectiveness of quantitative equidistribution techniques in answering questions in unlikely intersections, particularly with a view toward uniform results. For example, they are used by DeMarco, Krieger and Ye \cite{DKY} to obtain a uniform Manin-Mumford bound for a family of genus two curves. 

In this paper we develop the tools to apply quantitative equidistribution to some cases of Conjecture~\ref{conj:Ih}, and in particular make Petsche's \cite{P} result effective when $\rat$ is a polynomial. The bounds we obtain depend substantially on the $v$-adic distance of the given totally Fatou, non-preperiodic point $\alpha$ to the nearest preperiodic point for $\pol$, at each place $v \in S$. Bounding these distances from below would make quantitative the known discreteness of Fatou preperiodic points for a rational map (see for example \cite[Proposition~2]{P}). It is moreover interesting to study how these distances vary within a family of rational maps, and the extent to which we can obtain uniform bounds. Using various local methods, we do this in the case of a unicritical polynomial $\unipol(z)=z^2+c$, when $\alpha=0$ is the critical point. For parameters $c$ lying outside a small region, we obtain a bound for the number of prepriodic points for $\unipol$ which are $S$-integral with respect to 0, depending only on the number of primes for which $c$ is non-integral.

We also relate Conjecture~\ref{conj:Ih} to bounding the number of $S$-units in the doubly indexed sequence $\{ \rat^n(\alpha)-\rat^m(\alpha) \}_{n > m \geq 0}$. This generalises the problem of bounding the number of $S$-units in orbits of $\pol$, which has been considered in \cite{KLS}.

Throughout the paper, we will use the following notation. Let $M_K$ denote the set of places of $K$, and write $M_K^\infty$ and $M_K^0$ respectively for the sets of archimedean and non-archimedean places. For each $v \in M_K$, we normalise an absolute value $| \cdot |_v$, so that any $x \in K \setminus \{ 0 \}$ satisfies the \emph{product formula}
$$
\prod_{v \in M_K} |x|_v = 1.
$$
For each $v \in M_K$, let $\C_v$ be a completion of an algebraic closure of $K_v$, and let $D(a,r) := \{ z \in \C_v : |z-a|_v < r \}$ and $\overline D(a,r) := \{ z \in \C_v : |z-a|_v \leq r \}$ denote respectively the open and closed disks in $\C_v$ centred at $a \in \C_v$ of radius $r \geq 0$.

\subsection{Main results} 

For a rational function $\rat \in K(z)$ and a point $\alpha \in \overline K$, define
\begin{equation} \label{eq:defrelsupreper}
\relsupreper{\rat}{S}{\alpha} := \{ \beta \in \mathrm{PrePer}(\rat, \overline K) : \beta \text{ is $S$-integral relative to } \alpha \}
\end{equation}
and recall that our main goal is to obtain upper bounds for this quantity. We obtain the following.

\begin{theorem} \label{thm:main}
Let $K$ be a number field, and let $\pol \in K[z]$ be a polynomial of degree $d \geq 2$. Let $\canheight{\pol}$ denote the canonical height with respect to $\pol$ (see \textsection \textsection \ref{subsec:Heights}), and let $S$ be a finite set of places of $K$ containing all the archimedean ones, as well as all the places of bad reduction for $\pol$. Suppose that $\alpha \in \overline K$ is not preperiodic under $\pol$, and is totally Fatou at all places $v \in M_K$. 


Then there exists an effectively computable constant
$$
P = P \left( \pol, \alpha, S \right)
$$
such that
$$
|\relsupreper{\pol}{S}{\alpha}| \leq P.
$$
\end{theorem}

It is likely that one can obtain this result without restricting to the case where $\rat$ is a polynomial. However, the decomposition of the canonical height associated to $\rat$ into local heights, which arises when $\rat$ is a polynomial, simplifies the discussion greatly when compared to the general case, where one must analyse dynamical Arakelov-Green's functions appearing for example in \cite{BR1}.

For $\rat \in K(z)$ and $\alpha \in \overline K$, define
\begin{equation} \label{eq:defdeltav}
\delta_{\rat, v}(\alpha) := \min \left \{ 1, \inf_{z \in \mathrm{PrePer}(\rat,\C_v)} |z-\alpha|_v \right \}
\end{equation}
to be the $v$-adic distance of $\alpha$ to the nearest preperiodic point for $\rat$. 

\begin{remark}
Since the set of preperiodic points in the Fatou set of $\pol$ is discrete (see \cite[Proposition~2]{P}), the condition of $\alpha$ being non-preperiodic and totally Fatou at a place $v$ is equivalent to having $\delta_{\pol,v}(\alpha) > 0$.
\end{remark}

We show that in Theorem~\ref{thm:main}, the dependence of the constant $P$ on $\pol$ and $\alpha$ is limited to the canonical height $\canheight{\pol}(\alpha)$ (in fact, the constant improves as $\canheight{\pol}(\alpha)$ grows), the quantities $\delta_{\pol,v}(\alpha)$ for $v \in S$, and certain constants of H\"{o}lder continuous potentials for the invariant probability measure associated to $\pol$ at each place (see Theorem~\ref{thm:main2}). 

In the case of a unicritical polynomial $\unipol(z)=z^2+c$, when $\alpha=0$ is the critical point, we are able to explicitly estimate the dependence on $\unipol$ (especially the quantities $\delta_{\unipol,v}(0)$), and in most instances relate it to the canonical height $\canheight{\unipol}(0)$. This allows us to obtain the following, more uniform, result for parameters $c$ whose conjugates all lie outside a small region about the boundary of the Mandelbrot set (see \textsection \textsection \ref{subsec:Mandelbrot}). Note that in this case, we do not need to enlarge $S$ to contain the places of bad reduction for $\unipol$.

\begin{theorem} \label{thm:UniformMain}
Let $K$ be a number field, $S$ a finite set of places of $K$ containing the archimedean ones, and let $\unipol(z)=z^2+c$, $c \in K$. Suppose that the critical point $0$ is not preperiodic under $\unipol$. Further, let $\varepsilon > 0$, $t \geq 1$ and suppose that for all archimedean places $v \in M_K^\infty$, either $\loccanheight{v}{\unipol}(0) \geq \varepsilon$ or $c$ lies in a hyperbolic component of the Mandelbrot set of period at most $t$. Then
$$
|\relsupreper{\unipol}{S}{0}| \ll 1,
$$
where the implied constant depends only on $\varepsilon,t,K,S$ and the number of places of bad reduction for $\unipol$ (see Theorem~\ref{thm:UniformMainExplicit} for explicit values). 

\end{theorem}

Note that Theorem~\ref{thm:UniformMain} can be extended to $\unipol(z)=z^d+c$ for arbitrary $d \geq 2$, but we restrict to the case $d=2$ to avoid messy computations. Also, for parameters $c$ outside the Mandelbrot set, nothing special about the critical point is used in the proof of Theorem~\ref{thm:UniformMain}. It is therefore possible to obtain an analogue of Theorem~\ref{thm:UniformMain} with the critical point 0 replaced by an arbitrary non-preperiodic point $\alpha$, for parameters $c$ lying outside the corresponding bifurcation locus.

\begin{remark}
For a parameter $c \in K$ and an archimedean place $v$, if $c$ lies in the $v$-adic Mandelbrot set, then $\loccanheight{v}{\unipol}(0) = 0$. For $c$ outside the Mandelbrot set, it would be of interest to give a uniform estimate of the $v$-adic escape rate $\loccanheight{v}{\unipol}(0)$ in terms of the distance to the Mandelbrot set. However, this seems untractable currently, as it would solve in the affirmative the well-known open problem of the Mandelbrot set being locally connected; see \cite[Chapter~VIII,~Theorem~4.2]{CG}.
\end{remark}



Theorem~\ref{thm:UniformMain} produces the following completely uniform bound when $c$ is an integer.

\begin{corollary} \label{cor:IntMain}
Let $K = \Q$ and let $\unipol(z) = z^2 + c$ and $c \in \Z$ such that $0$ is not preperiodic for $\unipol$ i.e. $c \neq 0,-1,-2$. Let $S$ be a finite set of places of $\Q$ containing the archimedean place $\infty$. Then
$$
| \relsupreper{\unipol}{S}{0} | \ll 1,
$$
where the implied constant depends only on $S$ (see Corollary~\ref{cor:IntMain2} for explicit values).
\end{corollary}

Finally, we state our result on $S$-units in dynamical sequences. Let $\cO_S$ and $\cO_S^*$ denote respectively the sets of $S$-integers and $S$-units in $K$. In \cite{Si3}, Silverman showed that for any $\rat \in K(z)$ of degree at least two such that $\rat^2 \notin K[z]$, and any $\alpha \in K$, $\{ \rat^n(\alpha) : n \geq 0 \} \cap \cO_S$ is finite. However, the size of this intersection cannot be bounded independent of $\rat$, even if we restrict to functions of a fixed degree. Such a bound (i.e. depending only on $|S|$ and $\deg(\rat)$) may be possible if we replace $\cO_S$ with $\cO_S^*$. This was conjectured in \cite{KLS}, and proved in the case of a monic polynomial $\pol$ with coefficients in $\cO_S$. We proceed to consider, in place of the orbit of $\alpha$, the doubly-indexed sequence $\{ \rat^n(\alpha) - \rat^m(\alpha) \}_{n > m \geq 0}$. It is interesting to observe that we can relate the number of $S$-units in this sequence to the number of $\rat$-preperiodic points which are $S$-integral relative to $\alpha$.

\begin{theorem} \label{thm:SunitBound}
Let $K$ be a number field, let $\pol \in K[z]$ and let $\alpha \in K$. Let $S$ be a finite set of places of $K$ containing all the archimedean ones, and let
$$
\disunit{\pol}{S}{\alpha} : = \{ \pol^n(\alpha) - \pol^m(\alpha) \}_{n > m \geq 0}.
$$
Then
$$
|\disunit{\pol}{S}{\alpha} \cap \cO_S^*| \leq \frac{(2 |\relsupreper{\pol}{S}{\alpha}|+1)^2-1}{8},
$$
where we recall $\relsupreper{\pol}{S}{\alpha}$ is defined in \eqref{eq:defrelsupreper}.
\end{theorem}

\begin{remark}
In particular, for $\unipol(z) = z^2 + c$, $c \in \Z$, Corollary~\ref{cor:IntMain} implies
$$
| \disunit{\unipol}{S}{0} \cap \cO_S^* | \ll 1,
$$
with the implied constant depending only on $S$. 
\end{remark}

\subsection{Proof strategy and outline of paper} 

In Section~\ref{sec:Prelim}, we introduce notation and results requisite for potential theory on the Berkovich projective line at each place of $K$, as well as various useful notions and propositions relevant to heights of algebraic numbers. This allows us to proceed in Section~\ref{sec:QuantEqui} to make explicit the constants in Favre and Rivera-Letelier's quantitative equidistribution theorem \cite[Theorem~7]{FRL} for the adelic measure associated to a polynomial $\pol$, and in particular for the unicritical polynomial $\unipol(z)=z^2+c$. To prove Theorem~\ref{thm:main}, we assume on the contrary that the set $\cP:=\relsupreper{\pol}{S}{\alpha}$ of preperiodic points for $\pol$ which are $S$-integral relative to $\alpha$ is very large. The idea is to then test these preperiodic points against an extension of $\log |z-\alpha|_v$ at each place $v \in S$ using the equidistribution theorem. Then the quantity
$$
\Gamma := \frac{1}{|\cP|} \sum_{v \in M_K} \sum_{z \in \cP} \log | z - \alpha |_v,
$$
which can be shown to be 0 using the product formula, is very close to the canonical height of $\alpha$ with respect to $\pol$, and is hence positive (since $\alpha$ is not preperiodic), a contradiction. The only problem with this approach is that our candidate test function has a logarithmic singularity at (embeddings of) $\alpha$. As done by Petsche in \cite{P}, we get around this by making the assumption that $\alpha$ is totally Fatou at all places. In this case, since at each place the measure associated to $\pol$ is supported on its Julia set, and the preperiodic points of $\pol$ are discrete in the Fatou set, it is possible to construct a suitable truncation function which is sufficiently regular to apply our quantitative equidistribution theorem, and agrees with $\log |z-\alpha|_v$ at all preperiodic points, and on the support of the measure associated to $\pol$. We construct this truncation function in Section~\ref{sec:Trunc}, essentially by cutting out a disk of radius $\delta_{\pol,v}(\alpha)$. We then give the proof of Theorem~\ref{thm:main} in Section~\ref{sec:MainProof}. 


Most of the rest of the paper concerns a unicritical polynomial $\unipol(z)=z^2+c$, and is devoted to obtaining lower bounds for $\delta_v = \delta_{\unipol,v}(0)$ of an appropriate form at each relevant place, in order to eventually prove Theorem~\ref{thm:UniformMain} in Section~\ref{sec:Unif}. In Section~\ref{sec:arch} we deal with the archimedean places $v$, treating separately two different cases. Firstly, when the parameter $c$ lies outside the $v$-adic Mandelbrot set, we use properties of the geometry of the Julia set of $f$, as well as the H\"{o}lder continuity of its Green function to bound $\delta_v$ in terms of the $v$-adic escape rate $\loccanheight{v}{\unipol}(0)$.  When $c$ lies inside a hyperbolic component of the $v$-adic Mandelbrot set, we use distortion theorems from complex analysis together with uniformisation results from complex dynamics to bound $\delta_v$ in terms of the multiplier and the period of the component. We then remove the dependence on the multiplier using properties of the critical height found in \cite{I2}. In Section~\ref{sec:nonarch}, we apply techniques from non-archimedean analysis and dynamics to bound $\delta_v$ when $v$ is a finite place of good reduction.


Finally, we prove Theorem~\ref{thm:SunitBound} in Section~\ref{sec:Sunit} using an elementary argument, together with a lower bound on the number of distinct roots of $\pol^n(z)-\pol^m(z)$, $n > m \geq 0$.

\subsection{Acknowledgements} The author is indebted to his supervisor Holly Krieger for her support and guidance on this project, and to Rob Benedetto for useful discussions relevant to \textsection \ref{sec:nonarch}. 

\section{Preliminaries} \label{sec:Prelim}

\subsection{Reduction of rational maps} \label{subsec:Reduction}

Given a non-archimedean place $v$ of $K$, denote the \emph{ring of integers} and \emph{residue field} of $K$ with respect to $v$ respectively by $\cO_v$ and $k_v$. Given a rational function $\rat: \P^1 \to \P^1$ of degree $d$ defined over $K$, represented as
$$
\rat([x,y]) = [g(x,y),h(x,y)],
$$
where $g,h \in \cO_v[x,y]$ are homogeneous polynomials of degree $d$ with no common irreducible factors in $K[x,y]$ and at least one coefficient of $g$ or $h$ has $v$-adic absolute value 1, we say that $\rat$ has \emph{good reduction} at $v$ if the reductions of $f$ and $g$ modulo $v$ are relatively prime in $k_v[x,y]$. Otherwise we say that $\rat$ has \emph{bad reduction} at $v$. Let $M^{0,\rat}_{K,\mathrm{good}}$ denote the set of places of $K$ at which $\rat$ has good reduction, and let $M^{0,\rat}_{K,\mathrm{bad}} := M^0_K \setminus M^{0,\rat}_{K,\mathrm{good}}$. Note that the polynomial $\unipol(z) = z^d+c$, $c \in K$ has good reduction at $v$ if and only if $|c|_v \leq 1$.

\subsection{Fatou and Julia sets} \label{subsec:FatouJulia} 
Let $v \in M_K$. The \emph{Fatou set} $\cF_v(\rat)$ associated to a rational function $\rat$ is the largest open subset of $\P^1(\C_v)$ such that $\{ \rat^n \}_{n=1}^\infty$ is equicontinuous with respect to the spherical metric at all $z \in \cF_v(\rat)$. The \emph{Julia set} $\cJ_v(\rat)$ is the complement of the Fatou set. In the case of a polynomial $\pol$, $\cJ_v(\pol)$ is the boundary of the \emph{filled Julia set} 
$$
\cK_v(\pol) := \{ z \in \P^1(\C_v) : |\pol^n(z)| \not \to \infty \}
$$
for $\pol$. If $v \in M_K^0$ is a place at which $\rat$ has good reduction, then $\cJ_v(\rat) = \emptyset$. Note also that $\cJ_v(\rat)$ is contained in the closure of the set of all periodic points for $\rat$. In particular, $D(\alpha, \delta_{\rat,v}(\alpha)) \cap \cJ_v(\rat) = \emptyset$. For details about the Fatou and Julia sets, see \cite{B} \textsection \textsection 1.6.2 and Chapter~5. 

\subsection{The Mandelbrot set} \label{subsec:Mandelbrot} 
The \emph{Mandelbrot set} (see \cite{DH}) $\cM_d \subset \C$ is defined as
\begin{align*}
\cM & = \{ c \in \C : |\unipol^n(0)| \text{ is bounded} \} \\
& = \{ c \in \C : \text{the Julia set of $\unipol$ is connected} \},
\end{align*}
where $\unipol(z)=z^2+c$. We say that a connected component $\cH$ of the interior of $\cM$ is a \emph{hyperbolic component} of period $t$ if for all parameters $c \in \cH$, the map $\unipol(z)=z^2+c$ has an attracting periodic cycle of period $t$. Such a cycle, if it exists, must in fact be unique. For each archimedean place $v$ of $K$, by identifying $\C_v$ with $\C$, we obtain a $v$-adic Mandelbrot set $\cM_v \subset \C_v$.

\subsection{The Riemann sphere} \label{subsec:RiemannSphere} At an archimedean place $v$ of $K$, we identify $\P^1(\C_v)$ with $\P^1(\C)$, which we consider as a Riemannian 2-manifold by equipping it with the typical atlas and the Fubini-Study metric (see for example \cite[\textsection 4.1]{NC}), which induces an inner product $\langle \cdot, \cdot \rangle_g$ on the tangent space, and we let $\mu$ denote the volume element on $\P^1(\C)$. We extend the induced operators $\nabla$ and $\Delta$ (gradient and Laplacian respectively) to arbitrary distributions on $\P^1(\C)$ in the usual way (see for example \cite[\textsection 5.2]{NC}).




We say that a signed measure $\rho$ on $\P^1(\C)$ has \emph{continuous potentials} if there is a continuous function $h : \P^1(\C) \to \R$ such that $\Delta h = \rho - \lambda$ in the sense of distributions, where $\lambda$ is Lebesgue measure on the unit circle.

Given two signed measures $\rho$ and $\rho'$ on $\P^1(\C)$ such that $\log |z-w|$ is integrable with respect to $|\rho| \otimes |\rho'|$ on $\C \times \C \setminus \mathrm{Diag}$ (where $\mathrm{Diag} = \{ (z,z) : z \in \C \}$), we define the \emph{mutual energy} of $\rho$ and $\rho'$ by
$$
(\rho,\rho')=-\int_{\C \times \C \setminus \mathrm{Diag}} \log |z-w| d\rho(z) d \rho'(w).
$$

For any functions $f,h$ on $\P^1(\C)$ whose weak gradients exist, we define their \emph{Dirichlet form} by
$$
\langle f, h \rangle = \frac{1}{2\pi} \int_{\P^1(\C)} \langle \nabla f, \nabla h \rangle_g d \mu,
$$
which can be calculated in charts as (see \cite[Definition~2.17]{NC})
\begin{align} \label{eq:DirichletForm}
\langle f, h \rangle = \int_{\overline D(0,1)} & \left( \frac{\partial f_0}{\partial x} \frac{\partial h_0}{\partial x} + \frac{\partial f_0}{\partial y} \frac{ \partial h_0}{\partial y} \right) dx dy \notag \\
& + \int_{D(0,1)} \left( \frac{\partial f_1}{\partial x} \frac{\partial h_1}{\partial x} + \frac{\partial f_1}{\partial y} \frac{ \partial h_1}{\partial y} \right) dx dy,
\end{align}
where $f_0(x,y)=f([x+iy,1])$ and $f_1(x,y)=g([1,x+iy])$, and $[ \cdot, \cdot]$ are homogeneous coordinates on $\P^1(\C)$.

Let $H^1(\P^1(\C))$ denote the Sobolev space consisting of square integrable functions $\phi$ on $\P^1(\C)$ with weak gradients, satisfying $\langle \phi,\phi \rangle < \infty$ (equivalently in local coordinates, the weak partial derivatives of $\phi$ exist and are square integrable). We have the following Cauchy-Schwarz inequality

\begin{lemma} \label{lem:ArchCauchySchwarz}
For every function $\phi \in H^1(\P^1(\C))$ and every signed measure $\rho$ with $\rho(\P^1(\C)) = 0$ and such that $|\rho|$ has continuous potentials, we have
$$
\left| \int_{\P^1(\C)} \phi \, d \rho \right|^2 \leq \langle \phi, \phi \rangle (\rho, \rho).
$$
\end{lemma}

\begin{proof}
This is proved for $\phi \in \cC^1(\P^1(\C))$ in \cite[\textsection \textsection 5.5]{FRL} and \cite[Proposition~2.9]{NC}, but can easily be extended to $\phi \in H^1(\P^1(\C))$.
\end{proof}

Given a rational function $\rat$ of degree $d$, there is a canonical probability measure $\rho_{\rat,v}$ on $\P^1(\C_v)$, called the \emph{equilibrium measure} or \emph{measure of maximal entropy} associated to $\rat$ at $v$, whose support is the Julia set $\cJ_v(\rat)$, and which satisfies $\rat_*( \rho_{\rat,v}) = \rho_{\rat,v}$ and $\rat^*(\rho_{\rat,v}) = d \cdot \rho_{\rat,v}$ (see for example \cite{Br}).

\subsection{The Berkovich projective line} \label{subsec:Berk}

We recall details of this construction as presented in \cite[\textsection 3]{FRL} and \cite[\textsection 1]{RW}. For each place $v$ of $K$, the \emph{Berkovich affine line} $\BA(\C_v)$ is defined to be the set of multiplicative seminorms $\zeta = \| \cdot \|_\zeta$ on the polynomial ring $\C_v[T]$ that extend the absolute value $| \cdot |_v$ on $\C_v$, together with the weakest topology such that for every $f \in \C_v[T]$, the function $\zeta \mapsto \| f \|_\zeta$ is continuous (called the \emph{Gel'fand topology}). At archimedean places, we have $\BA(\C_v) \cong \C_v$ (this can be deduced, for example, from the Gelfand-Mazur theorem). Let $v$ be a non-archimedean place. Then the \emph{Berkovich projective line} $\BP(\C_v) = \BA(\C_v) \cup \{ \infty \}$ is a Hausdorff, compact and uniquely path connected topological space containing $\P^1(\C_v)$ as a dense subspace (see \cite[Chapter~6]{B}, or \cite[\textsection 3]{FRL} for details of its construction and properties). In particular, $\BP(\C_v)$ has the structure of an $\R$-tree (see \cite[Appendix~B]{B}). Given a disk $D(a,r) \subseteq \C_v$, there is a corresponding point $\zeta(a,r) \in \BA(\C_v)$, whose seminorm is given by
$$
\| f \|_{\zeta(a,r)} = \sup \{ |f(z)|_v : z \in D(a,r) \}.
$$ 
The point $\zeta_G := \zeta(0,1)$ is called the \emph{Gauss point}.

By Berkovich's classification theorem \cite[Theorem~6.9]{B}, all points of $\BA(\C_v)$ either correspond to a disk $D(a,r)$ (and are respectively called Type I, II, and III points when $r=0$, $r \in |\C_v^\times|$ and $r \notin |\C_v^\times|$), or to (a cofinal equivalence class of) a sequence of nested disks $D(a_n,r_n)$ with empty intersection. A point $\zeta \in \BA(\C_v)$ of the latter kind, called Type IV, has seminorm given by
$$
\| f \|_\zeta = \lim_{n \to \infty} \| f \|_{\zeta(a_n,r_n)}.
$$
For $\zeta \in \BA(\C_v)$, the \emph{absolute value} of $\zeta$ is $|\zeta|:= \| T \|_\zeta$ and the \emph{diameter} of $\zeta$ is
$$
\diam(\zeta) := \inf \{ \| T - a \|_\zeta : a \in \C_v \}.
$$

As $\BP(\C_v)$ is uniquely path-connected, given any $x,y \in \BP(\C_v)$ we will denote by $[x,y]$ the unique path between $x$ and $y$, and denote by $(x,y)$, $[x,y)$ and $(x,y]$ the corresponding open and half-open paths. Moreover, if we fix a point $\zeta \in \BP(\C_v)$, the set $[x,\zeta] \cap [y,\zeta] \cap [x,y]$ consists of a single point (see \cite[Proposition~6.35]{B}) which we denote by $x \vee_\zeta y$. We write $\vee$ for $\vee_\infty$ and $\vee_G$ for $\vee_{\zeta_G}$.

\subsection*{Distances on $\BP(\C_v)$}

The set $\H_v : = \BP(\C_v) \setminus \P^1(\C_v)$ is called the \emph{hyperbolic space} over $\C_v$, and admits the following metric, called the \emph{logarithmic path distance}:
\begin{equation} \label{eq:defPathDistance}
d_{\H_v}(x,y) = 2 \log \diam(x \vee y) - \log \diam(x) - \log \diam(y).
\end{equation}
We can extend $d_{\H_v}$ to a singular metric on $\BP(\C_v)$ by declaring that for $x \in \P^1(\C_v)$ and $y \in \BP(\C_v)$, we have $d_{\H_v}(x,y) = \infty$ if $x \neq y$ and 0 otherwise.

We can also define the diameter of a point $x \in \BP(\C_v)$ with respect to the Gauss point $\zeta_G$ by
\begin{equation} \label{eq:defGaussDiam}
\diam_G(x) = \exp(-d_{\H_v}(\zeta_G,x)),
\end{equation}
and subsequently obtain a metric $\ud$ proportional to that defined by Favre and Rivera-Letelier \cite[\textsection \textsection 4.7]{FRL} on $\BP(\C_v)$, given by
\begin{equation} \label{eq:FRLDist}
\ud(x,y) = 2 \, \diam_G(x \vee_G y) - \diam_G(x) - \diam_G(y).
\end{equation}
This satisfies $0 \leq \ud(x,y) \leq 2$ and $\ud(x,y) = 2 \| x,y \|$ when $a,b \in \P^1(\C_v)$ (see \cite[Proposition~1.1]{RW}), where $\| \cdot, \cdot \|$ denotes the spherical metric on $\P^1(\C_v)$. Both $d_{\H_v}$ and $\ud$ are additive on paths in the sense that if $z$ is any point in $[x,y]$, we have $d_{\H_v}(x,y) = d_{\H_v}(x,z) + d_{\H_v}(z,y)$, and similarly for $\ud$.

The \emph{Gromov product} (with the Gauss point $\zeta_G$ as a base point) is the function $\langle \cdot, \cdot \rangle_G : \BP(\C_v) \times \BP(\C_v) \to [0,\infty]$ given by
\begin{equation} \label{eq:GromovDef}
\langle x, y \rangle_G = d_{\H_v}(x \vee_G y,\zeta_G).
\end{equation}

\subsection*{Potential Theory on $\BP(\C_v)$}

Every rational function $\rat: \P^1(\C_v) \to \P^1(\C_v)$ has a unique continuous extension to $\BP(\C_v)$ \cite[\textsection 7.1]{B}. We also have the notion of Julia set of $\rat$ on the Berkovich projective line. We say an open set $U \subseteq \BP(\C_v)$ is \emph{dynamically stable} under $\rat$ if $\bigcup_{n \geq 0} \rat^n(U)$ omits infinitely many points of $\BP(\C_v)$. The \emph{Berkovich Fatou set}, which we denote $\cF_{v,\mathrm{Berk}}(\rat)$, is the subset of $\BP(\C_v)$ consisting of all points $x \in \BP(\C_v)$ having a dynamically stable neighbourhood. Again, the \emph{Berkovich Julia set} is given by $\cJ_{v,\mathrm{Berk}}(\rat) = \BP(\C_v) \setminus \cF_{v,\mathrm{Berk}}(\rat)$. When $v$ is a place of good reduction for $\rat$, we have $\cJ_{v,\mathrm{Berk}}(\rat) = \{ \zeta_G \}$.



Denote by $M_{\mathrm{Rad}}$ the set of signed Radon measures on $\BP(\C_v)$, and let $\rho \in M_{\mathrm{Rad}}$. Then the \emph{potential} of $\rho$ based at $\zeta_G$ is the function $u_\rho : \H_v \to \R$ given by
$$
u_\rho(x) := - \rho(\BP(\C_v)) - \int_{\BP(\C_v)} \langle x, y \rangle_G \, d \rho(y).
$$
Note that the potential of a delta measure at a point $y$, which we denote $[y]$, is given by $u_{[y]}(x) = -1 - \langle x,y \rangle_G$. Let $\cQ$ denote the space of potentials 
$$
\cQ := \{ f: \H_v \to \R : f = u_\rho \text{ for some } \rho \in M_{\mathrm{Rad}} \}.
$$
This is a vector space which contains all functions of the form $\langle \cdot, y \rangle_G$. With this framework, one can define a non-archimedean Laplacian operator on $\BP(\C_v)$ and develop the relevant notions in potential theory. Namely, define the function $\Delta : \cQ \to M_{\mathrm{Rad}}$ by
$$
\Delta u_\rho := \rho - \rho(\BP(\C_v)) \cdot [\zeta_G].
$$
Importantly, we have
\begin{equation} \label{eq:LaplacGromov}
\Delta \langle \cdot, y \rangle_G = [\zeta_G] - [y].
\end{equation}
For $\rho \in M_{\mathrm{Rad}}$, we say that $\rho$ has \emph{continuous potentials} if the associated potential function $u_\rho$ extends to a function on $\BP(\C_v)$ which is continuous with respect to the Gel'fand topology.

Given a pair of measures $\rho, \rho' \in M_{\mathrm{Rad}}$, such that $\log \diam ( \cdot \vee \cdot )$ is integrable with respect to $\rho \otimes \rho'$ on $\BA(\C_v) \times \BA(\C_v) \setminus \mathrm{Diag}$ (where $\mathrm{Diag} = \{ (z,z) : z \in \C_v \}$ is the Type I diagonal), we define the \emph{mutual energy} of $\rho$ and $\rho'$ to be
\begin{equation} \label{eq:defMutualEnergy}
(\rho,\rho') = - \int_{\BA(\C_v) \times \BA(\C_v) \setminus \mathrm{Diag}} \log \diam( x \vee y) \, d \rho(x) d \rho'(y).
\end{equation}
Any rational map $\rat : \P^1(\C_v) \to \P^1(\C_v)$ of degree $d \geq 2$ extends naturally to a map $\rat: \BP(\C_v) \to \BP(\C_v)$. As in the archimedean case, we can associate to $\rat$ an \emph{equilibrium measure} i.e. a probability measure $\rho_{\rat,v}$ on $\BP(\C_v)$ which satisfies $\rat_*(\rho_{\rat,v}) = \rho_{\rat,v}$ and $\rat^*(\rho_{\rat,v}) = d \cdot \rho_{\rat,v}$ (see \cite[Theorem~13.37]{B} or \cite[\textsection \textsection 6.1]{FRL} for details of its construction). Once again, we note that $\mathrm{supp}(\rho_{\rat,v}) = \cJBv(\rat)$.

We say a continuous function $\phi : \BP(\C_v) \to \R$ is $\cC^k$ for $k \geq 1$ if it is locally constant of a finite and closed subtree $\cT \subset \H_v$, and $\phi$ is $\cC^k$ on $\cT$, where for $x \in \cT$ we can define the derivative $\partial \phi(x)$ as the left derivative of $\phi$ at $x$ on the path $[\zeta_G, x]$, parameterised by the distance $d_{\H_v}$, and $\partial \phi(x)=0$ for $x \notin \cT$. Now, for a $\cC^1$ function $\phi$, set
$$
\langle \phi, \phi \rangle := \int_{\BP(\C_v)} (\partial \phi)^2 d \mu,
$$
where $\mu$ is the measure on $\BP(\C_v)$ which is null on $\P^1(\C_v)$, and coincides with Hausdorff measure of dimension 1 on $\H_v$ with respect to the distance $d_{\H_v}$. As in the archimedean case, we have a kind of Cauchy-Schwarz inequality \cite[\textsection \textsection 5.5]{FRL}.

\begin{lemma} \label{lem:nonArchCauchySchwarz}
Let $\phi$ be a $\cC^1$ function on $\BP(\C_v)$, and let $\rho \in M_{\mathrm{Rad}}$ have continuous potentials and satisfy $\rho(\BP(\C_v)) = 0$. Then
$$
\left| \int_{\BP(\C_v)} \phi \, d \rho \right|^2 \leq \langle \phi, \phi \rangle (\rho, \rho).
$$
\end{lemma}

For $\varepsilon > 0$, let $\pi_\varepsilon : \BP(\C_v) \to \BP(\C_v)$ be the map which sends $x \in \BP(\C_v)$ to the unique point $y$ in $[x,\infty]$ with $\diam(y) = \max \{ \diam(x),\varepsilon \}$. We define the \emph{regularisation} of a measure $\rho \in M_{\mathrm{Rad}}$ to be $\rho_\varepsilon := (\pi_\varepsilon)_* \rho$.

\subsection{Adelic measures and heights} \label{subsec:AdelMeasHei} When $v$ is an infinite place of $K$, let $\lambda_v$ denote the probability measure proportional to Lebesgue measure on the unit circle $S^1 \subset \P^1(\C_v)$. When $v$ is finite, let $\lambda_v$ denote the point mass at the Gauss point in $\BP(\C_v)$.

An \emph{adelic measure} $\rho = \{ \rho_v \}_{v \in M_K}$ consists of a probability measure $\rho_v$ on $\BP(\C_v)$ with continuous potentials for each place $v \in M_K$, such that $\rho_v = \lambda_v$ for all but finitely many places $v$. To an adelic measure $\rho$ we can associate an \emph{adelic height}, given by
\begin{equation} \label{eq:adelHeight}
h_\rho(F) := \frac{1}{2} \sum_{v \in M_K} ( [F] - \rho_v, [F]-\rho_v )_v,
\end{equation}
for any finite set $F \subset \overline K$ invariant under $\Gal(\overline K/K)$. Note that in \cite[Definition~1.2]{FRL}, $\llrrparen{ \cdot , \cdot }_v$ is used in place of $( \cdot , \cdot )_v$, denoting a normalisation of the mutual energy, but this is consistent with the way we have initially normalised our absolute values on $K$. For $\alpha \in \P^1(\overline K)$, we let $h_\rho(\alpha) = h_\rho(F)$ where $F$ is the orbit of $\alpha$ under the action of $\Gal(\overline K/K)$. When $\rho = \{ \lambda_v \}_{v \in M_K}$, we have $h_\rho = h$, the usual height, whose definition we recall in the next subsection. We say that an adelic measure $\rho$ has \emph{H\"{o}lder continuous potentials} if for each place $v$, there exists a H\"{o}lder continuous (with respect to the spherical metric in the archimedean case, and the metric $\ud$ defined in \eqref{eq:FRLDist} in the non-archimedean case) function $g_v : \BP(\C_v) \to \R$ with $\Delta g_v = \rho_v - \lambda_v$. It turns out that in the archimedean case, it suffices (and is more convenient) to make some calculations with the Euclidean distance, rather than the spherical one. Hence, given an adelic measure $\rho$ with H\"{o}lder continuous potentials, we will choose $g_v$ such that there exist constants $C_v \geq 0$ and $\kappa_v \leq 1$ with
\begin{equation} \label{eq:defHCP}
\begin{cases} |g_v^*(z)-g_v^*(w)| \leq  C_v |z-w|_v^{\kappa_v} \text{ for all } z,w \in \C_v, & \text{if } v \in M_K^\infty, \\
|g_v(z)-g_v(w)| \leq  C_v \ud(z,w)^{\kappa_v} \text{ for all } z,w \in \BP(\C_v), & \text{if } v \in M_K^0, \end{cases}
\end{equation}
where for $v \in M_K^\infty$, $g_v^*(z) = g_v(z) + \log^+|z|_v$ so that $\Delta g_v^* = \rho_v$.

Favre and Rivera-Letelier \cite[Theorem~4]{FRL} show that for a rational function $\rat$, $\rho_{\rat} = \{ \rho_{\rat,v} \}_{v \in M_K}$ is an adelic measure with H\"{o}lder continuous potentials, where $\rho_{\rat,v}$ is the equilibrium measure of $\rat$ at $v$, and that the adelic height $h_{\rho_{\rat}}$ coincides with the canonical height $\canheight{\rat}$ for $\rat$. 

\subsection{Heights of algebraic numbers} \label{subsec:Heights} Recall that the \emph{(absolute logarithmic) height} of an algebraic number $\alpha$ is given by
$$
h(\alpha) = \sum_{v \in M_K} \log^+ |\alpha|_v,
$$
where $K$ is any number field containing $\alpha$, and $\log^+x = \log \max \{ 1, x \}$. The height is independent of the choice of number field, invariant under Galois conjugation, and satisfies the following properties (see \cite[\textsection \textsection 1.5]{BG}). For algebraic numbers $\alpha_1,\ldots,\alpha_r$, we have
\begin{equation} \label{eq:HeightSumInequality1}
h(\alpha_1 + \cdots + \alpha_r) \leq h(\alpha_1) + \cdots + h(\alpha_r) + \log r,
\end{equation}
and moreover
\begin{align*} 
h(\alpha_1) & = h( (\alpha_1 + \cdots + \alpha_r) - \alpha_2 - \cdots - \alpha_r ) \\
& \leq h(\alpha_1 + \cdots + \alpha_r ) + h(-\alpha_2) + \cdots + h(-\alpha_r) + \log r,
\end{align*}
and so
\begin{equation} \label{eq:HeightSumInequality2}
h(\alpha_1 + \cdots+ \alpha_r) \geq h(\alpha_1) - h(\alpha_2) - \cdots - h(\alpha_r)-\log r.
\end{equation}
For a non-zero algebraic number $\alpha$ and a rational number $\lambda$, we have $h(\alpha^\lambda) = |\lambda| h(\alpha)$. For $K$ a number field $S \subset M_K$ a finite set of places, and $\alpha \in K \setminus \{ 0 \}$, we have the fundamental inequality
\begin{equation} \label{eq:FundamentalHeightInequality}
-h(\alpha) \leq \sum_{v \in S} \log |\alpha|_v \leq h(\alpha).
\end{equation}

Now, we can associate to a given rational function $\rat : \P^1 \to \P^1$ a \emph{canonical height} (see \cite[\textsection 3.4]{Si}) $\canheight{\rat}$ given by
$$
\canheight{\rat}(\alpha) = \lim_{n \to \infty} \frac{1}{d^n} h(\rat^n(\alpha)).
$$
This satisfies $\canheight{\rat}(\alpha) = h(\alpha) + O(1)$, and $\canheight{\rat}(r(\alpha)) = d \canheight{\rat}(\alpha)$. When $\rat=\pol$ is a polynomial, $\canheight{\pol}$ has the following decomposition into \emph{local canonical heights} (or \emph{$v$-adic escape rates}) given by
$$
\canheight{\pol}(\alpha) = \sum_{v \in M_K} \loccanheight{v}{\pol}(\alpha),
$$
where $K$ is a number field containing $\alpha$, and
$$
\loccanheight{v}{\pol}(\alpha) = \lim_{n \to \infty} \frac{1}{d^n} \log^+ |\pol^n(\alpha)|_v.
$$
Note that if $v$ is a place of good reduction for $R$, then we have $\loccanheight{v}{\pol}(\alpha) = \log^+ |\alpha|_v$. 

Also, we have the following formula for the local height (see the proof of \cite[Proposition~1.3]{FRL}):
\begin{equation} \label{eq:MahlerFormula}
\loccanheight{v}{\pol}(\alpha) = \int_{\BP(\C_v)} \log |z-\alpha|_v d \rho_{\pol,v},
\end{equation}
where $\rho_{\pol,v}$ is the equilibrium measure associated to $\pol$ at the place $v$, and for non-archimedean places, we mean
$$
\log |z-\alpha|_v := \log \diam(z \vee \alpha)
$$
for $z \in \BP(\C_v)$. Note that at archimedean places, this implies
\begin{equation} \label{eq:ArchPotLCH}
\Delta \loccanheight{v}{\pol} = \rho_{\pol,v}.
\end{equation}

\subsection{Canonical heights of unicritical polynomials} We now collect a number of useful properties of the local and global canonical heights associated to a unicritical polynomial.

\begin{lemma} \label{lem:BadReductHeight}
Let $\unipol(z) = z^2 + c$, $c \in K$, and suppose that $v$ is a place of bad reduction for $\unipol$, or equivalently $|c|_v > 1$. Then
$$
\loccanheight{v}{\unipol}(\alpha) = \log \max \{ |\alpha|_v, |c|_v^{1/2} \}
$$
for all $\alpha \in \C_v$ with $|\alpha|_v \neq |c|_v^{1/2}$.
\end{lemma}

\begin{proof}
This is an easy induction.
\end{proof}

\begin{remark} \label{rem:BadReductHeight}
Note that the above result shows that at a place $v$ of bad reduction for $\unipol(z)=z^2+c$, any point $\beta$ which is preperiodic for $\unipol$ must satisfy $|\beta|_v = |c|_v^{1/2}$. Similarly, if $v$ is a place of good reduction for $\unipol$, then preperiodic points $\beta$ for $\unipol$ must satisfy $|\beta|_v \leq 1$.
\end{remark}

For archimedean places $v$, the local canonical height is precisely the Green function of the filled Julia set of $\unipol$, and satisfies the following properties (these are proved in \cite[Lemma~2.3]{BD})

\begin{lemma} \label{lem:LocalHeightProps}
Let $\unipol(z)=z^2+c$, $c \in K$. Let $v$ be an archimedean place of $K$, and let $z \in \C_v$. Then
\begin{enumerate}
\item[(a)] When $|c|_v \geq 1$ or $|z|_v \geq |c|_v^{1/2}$ we have
$$
\loccanheight{v}{\unipol}(z) \leq \log 2 + \max \left \{ \frac{1}{2} \log |c|_v, \log |z|_v \right \}.
$$
\item[(b)] $\max \{ \loccanheight{v}{\unipol}(z), \loccanheight{v}{\unipol}(0) \} \geq \log |z|_v - \log 4.$
\end{enumerate}
\end{lemma}

Ingram \cite[Theorem~2]{I}, gives the following lower bound for the canonical height (with respect to $\unipol$) of a non-preperiodic point in terms of the height of the parameter $c$.

\begin{theorem} \label{thm:IngramCanHeightBound}
Let $c \in K$, let $s$ be the number of places $v$ of bad reduction for $\unipol(z)=z^2+c$ such that $v(c)=\exp(-|c|_v)$ is divisible by $2$, let $r \leq [K:\Q]$ be the number of distinct archimedean valuations on $K$, let

and let
$$
N = \frac{5^{r+s+1}-1}{2}.
$$
Then for any $\alpha \in K$, either $\unipol^i(\alpha) = \unipol^j(\alpha)$ for some $i \neq j < N$, or else
$$
\canheight{\unipol}(\alpha) \geq \frac{1}{2^{N+2}} (h(c) - 12\log 2).
$$
\end{theorem}




Note the following trivial bounds on the canonical height when the parameter $c$ is an integer.

\begin{lemma} \label{lem:IntCanHeightBounds}
Let $K = \Q$. For $c=1$, 
$$
\canheight{\unipol}(0) \geq \frac{1}{4} \log 2
$$
and for $c \in \Z$ such that $0$ is not preperiodic for $\unipol$ i.e. $c \neq -2,-1,0$,
$$
\canheight{\unipol}(0) \geq \frac{1}{4} h(c).
$$
\end{lemma}


Combining some of the above results gives the following.

\begin{corollary} \label{cor:CanHeightBound1}
With notation as in Theorem~\ref{thm:IngramCanHeightBound}, we have
$$
h(c) \leq 2^{N+2} \canheight{\unipol}(0) + 12 \log 2,
$$
Moreover, in the case where, $K = \Q$ and $c \in \Z$, we have $h(c) \leq 4 \canheight{\unipol}(0)$.
\end{corollary}

\begin{proof}
The result follows from Theorem~\ref{thm:IngramCanHeightBound}. The integer case follows from Lemma~\ref{lem:IntCanHeightBounds}.
\end{proof}

We also obtain a lower bound for the canonical height of any non-preperiodic point $\alpha$ for $\unipol$.

\begin{corollary} \label{cor:CanHeightLowBound}
For any $B > 0$, let $\cN(K,B)$ denote the number of elements of $K$ with height at most $B$. Then, for any $\alpha \in K$ which is not preperiodic for $\unipol$, we have
$$
\canheight{\unipol}(\alpha) \geq C_0 \max \{ 1, h(c) \},
$$
where
$$
C_0 \geq \min \left \{ \frac{1}{2^{N+3}}, \frac{2 \log 2}{2^{\cN(K, 25\log 2)}} \right \}.
$$
Moreover, when $\alpha = 0$, $K = \Q$ and $c \in \Z$, we can take
$$
C_0 \geq \frac{1}{4}.
$$
\end{corollary}

Note that $\cN(K,B)$ can be estimated. For example \cite[Theorem~1.1]{D} gives
$$
\log \cN(K,B) < B[K:\Q]^2 + \frac{17 B [K : \Q]^2 \log \log [K : \Q]}{\log [K : \Q]}
$$
for all $B > 2 [K : \Q]^{-1} \log [K : \Q]$ when $[K : \Q]$ is sufficiently large.

\begin{proof}
If $h(c) \geq 24 \log 2$, then immediately from Theorem~\ref{thm:IngramCanHeightBound} we have
$$
\canheight{\unipol}(\alpha) \geq \frac{h(c)}{2^{N+3}} = \frac{1}{2^{N+3}} \max \{ 1, h(c) \}.
$$
By \eqref{eq:HeightSumInequality1} and \eqref{eq:HeightSumInequality2},
\begin{equation*}
|h(\unipol(\beta)) - 2h(\beta) | \leq h(c) + \log 2,
\end{equation*}
for any $\beta \in K$. Hence, for any $\beta \in K$,
\begin{equation} \label{eq:CanHeightConst}
|\canheight{\unipol}(\beta) - h(\beta)| \leq h(c)+\log 2.
\end{equation}
(see for example the proof of \cite[Theorem~3.20]{Si}).

Suppose $h(c) < 24 \log 2$. Then \eqref{eq:CanHeightConst} with $\beta = \unipol^\ell(\alpha)$ gives
$$
2^\ell \canheight{\unipol}(\alpha) = \canheight{\unipol}(\unipol^\ell(\alpha)) \geq h(\unipol^\ell(\alpha)) - 25 \log 2
$$
for all $\ell \geq 1$. Since $\alpha$ is not preperiodic for $\unipol$, there exists $\ell \leq \cN(K, 25\log 2)$ such that
$$
h(\unipol^\ell(\alpha)) \geq 50\log 2,
$$
and hence
\begin{align*}
\canheight{\unipol}(\alpha) & \geq \frac{50 \log 2}{2^{\cN(K, 25\log 2)}} \\
& \geq \frac{2 \log 2}{2^{\cN(K, 25\log 2)}} \max \{1, h(c) \}.
\end{align*}
The integer case comes immediately from Lemma~\ref{lem:IntCanHeightBounds}.
\end{proof}

\section{Explicit quantitative equidistribution} \label{sec:QuantEqui}

For any $v \in M_K$ and any finite subset $F \subset \BP(\C_v)$, we denote by $[F]$ the probability measure
$$
[F] := \frac{1}{|F|} \sum_{z \in F} [z],
$$
where we recall $[z]$ is the Dirac measure at $z$. For $z \in \BP(\C_v)$ and $\varepsilon > 0$, if $v$ is archimedean let $[z]_\varepsilon$ denote the probability measure proportional to Lebesgue measure on the circle about $z$ of radius $\varepsilon$. Otherwise, denote by $\pi_\varepsilon : \BP(\C_v) \to \BP(\C_v)$ the map which sends a point $x$ to the unique point $y \in [x,\infty]$ such that $\diam(y) = \max \{ \diam(x),\varepsilon \}$, and set $[z]_\varepsilon = (\pi_\varepsilon)_* [z]$ (see \cite[\textsection \textsection 4.6]{FRL}).

For a finite subset $F \subset \BP(\C_v)$, define
$$
[F]_\varepsilon = \frac{1}{|F|} \sum_{z \in F} [z]_\varepsilon.
$$

Let us firstly make the constant in \cite[Theorem~7]{FRL} explicit in terms of the constants and exponents of H\"{o}lder continuity of the given measure.

\begin{theorem} \label{thm:QuantEquid}
Suppose $\rho = \{ \rho_v \}_{v \in M_K}$ is an adelic measure with H\"{o}lder continuous potentials $\{ g_v \}_{v \in M_K}$, with constants $\{ C_v \}_{v \in M_K}$ and exponents $\{ \kappa_v \}_{v \in M_K}$ as in \eqref{eq:defHCP}. Let $V = \{ v \in M_K : \rho_v \neq \lambda_v \}$, where $\lambda_v$ is defined as in \textsection \textsection \ref{subsec:AdelMeasHei}, and let $C, \kappa > 0$ be constants such that $C \geq \sum_{v \in V} C_v$ and $\kappa \leq \min_{v \in V} \kappa_v$. Then for any place $v \in M_K$, any test function $\phi$ on $\BP(\C_v)$ which is $\cC^1$ if $v$ is finite and otherwise Lipschitz on $\C_v$, and any finite $\Gal(\overline K/K)$-invariant subset $F$ of $\overline K$ with $|F| \geq (6C \kappa)/(|V|+1)$ we have
\begin{align*}
\left| \frac{1}{|F|} \sum_{\alpha \in F} \phi(\alpha) - \int_{\BP(\C_v)} \phi \, d \rho \right| & \leq \left( h_{\rho}(F) + \frac{2(|V|+1)}{\kappa} \frac{\log |F|}{|F|} \right)^{\frac{1}{2}} \langle \phi, \phi \rangle^{\frac{1}{2}} \\
& \qquad \qquad \qquad \qquad + \Lip(\phi) \left( \frac{|V|+1}{2C \kappa |F|} \right)^{\frac{1}{\kappa}}.
\end{align*}
\end{theorem}

\begin{proof}
Following the proof of \cite[Theorem~7]{FRL}, let $0 < \varepsilon < 1$ and let $F$ be a finite $\Gal(\overline K/ K)$-invariant subset of $\overline K$. Then for each place $v \in M_K$, we have
\begin{align*}
\sum_{w \in V \cup \{ v \}} & (  [F]_\varepsilon - \rho_w, [F]_\varepsilon - \rho_w)_w = \sum_{w \in V \cup \{ v \}} \Big( ([F]-\rho_w,[F]-\rho_w)_w \\
& + 2 \left( ([F],\rho_w)_w-([F]_\varepsilon,\rho_w)_w \right) + \left( ([F]_\varepsilon,[F]_\varepsilon)_w - ([F],[F])_w \right) \Big) \\
& \leq \sum_{w \in V \cup \{ v \}} \left( ([F]-\rho_w,[F]-\rho_w)_w + 2 \left( ([F],\rho_w)_w-([F]_\varepsilon,\rho_w)_w \right) - \frac{\log \varepsilon}{|F|} \right),
\end{align*}
where the inequality follows from \cite[Lemma~12]{F}, and \cite[Proposition~4.9]{FRL}. Now, let $w$ be an archimedean place so $\rho_w = \Delta g^*_w$ as in \eqref{eq:defHCP}. We have by \cite[Lemma~2.5]{FRL},
\begin{align*}
\left| ([F]_\varepsilon,\rho_w)_w - ([F],\rho_w)_w \right| & \leq \frac{1}{|F|} \sum_{z \in F} \left| ([z]_\varepsilon, \rho_w )_w - ([z],\rho_w)_w \right| \\
& = \frac{1}{|F|} \sum_{z \in F} \left| \int_{\C_w} g^*_w d([z]_\varepsilon-[z]) \right| \\
& \leq \frac{1}{|F|} \sum_{z \in F} \int_0^1 |g^*_w(z+\varepsilon e^{2\pi i \theta}) - g^*_w(z)| d\theta \leq C_w \varepsilon^{\kappa_w}.
\end{align*}
The analogous inequality holds for non-archimedean places by \cite[Lemma~4.9]{FRL}. Also, for all $w \notin V$ we have $\rho_w=\lambda_w$ and \cite[Lemma~5.4]{FRL} gives $([F]-\rho_w,[F]-\rho_w)_w \geq 0$, and so
\begin{align} \label{eq:EnergBd}
( [F]_\varepsilon - \rho_v, [F]_\varepsilon - \rho_v )_{v} & \leq h_{\rho}(F) + \sum_{w \in V \cup \{v\}} \left( 2C_w \varepsilon^{\kappa_w} - \frac{\log \varepsilon}{|F|} \right) \notag \\
& \leq 2 C \varepsilon^\kappa - \left( |V| + 1 \right) \frac{\log \varepsilon}{|F|}.
\end{align}
We choose
$$
\varepsilon = \left( \frac{|V|+1}{2 C \kappa |F|} \right)^{1/\kappa}
$$
in order to minimise the right hand side of \eqref{eq:EnergBd}. Then we have
$$
( [F]_\varepsilon - \rho_v, [F]_\varepsilon - \rho_v )_{v} \leq h_\rho(F) + \frac{|V|+1}{\kappa|F|} \left(1+ \log \frac{2C \kappa}{|V|+1} + \log |F| \right),
$$
Then, by Lemmas~\ref{lem:ArchCauchySchwarz} and \ref{lem:nonArchCauchySchwarz}, for our given test function $\phi$ we have
$$
\left| \int_{\BP(\C_v)} \phi \, d([F]_\varepsilon - \rho_v) \right|^2 \leq \langle \phi,\phi \rangle \left( h_{\rho}(F) + \frac{|V|+1}{\kappa|F|} \left(\log \frac{6C \kappa}{|V|+1} + \log |F| \right) \right).
$$
On the other hand,
\begin{align*}
\left| \int_{\BP(\C_v)} \phi \, d([F]_\varepsilon - [F]) \right| & \leq \mathrm{Lip}(\phi) \cdot \varepsilon \\
& = \mathrm{Lip}(\phi) \left( \frac{|V|+1}{2C \kappa |F|} \right)^{1/\kappa}.
\end{align*}
Combining these and recalling that $|F| \geq (6C\kappa)/(|V|+1)$ gives the result.
\end{proof}

Favre and Rivera-Letelier \cite[Proposition~6.5]{FRL} show that the adelic measure associated to a rational function has H\"{o}lder continuous potentials. We will give bounds for the constants $C_v$, $\kappa_v$ associated to $\rho_{\pol}$ for a polynomial $\pol \in K[z]$ (see \eqref{eq:defHCP}). At archimedean places, the H\"{o}lder continuity of the local canonical height is a classical result, and we follow an approach that can be found in \cite[Theorem~1.1]{K} and \cite[Chapter~VIII,~Theorem~3.2]{CG}.

\begin{lemma} \label{lem:HolArch}
Let $\pol \in K[z]$ have degree $d \geq 2$. Let $v$ be an archimedean place of $K$, and let $F \subset \C_v$ be a compact, convex set containing the $v$-adic filled Julia set $\cK_v(p)$ of $p$. Then for any $\varepsilon > 0$, we have
$$
|\loccanheight{v}{\pol}(z)-\loccanheight{v}{\pol}(w)|_v \leq C_{\pol,v}(\varepsilon)|z-w|_v^{\kappa_{\pol,v}(\varepsilon)},
$$ 
for all $z,w \in \C_v$, where
$$
\kappa_{\pol,v}(\varepsilon) = \frac{\log d}{\log \max \{ |\pol'(z)|_v: z \in F_\varepsilon \} }
$$
and
$$
C_{\pol,v}(\varepsilon) = 2d \varepsilon^{-\kappa_{\pol,v}(\varepsilon)} \max \{ \loccanheight{v}{\pol}(z) : z \in F_\varepsilon \}.
$$
Here $F_\varepsilon$ denotes the set of $z \in \C_v$ with $\mathrm{dist}(z,F) < \varepsilon$.
\end{lemma}

\begin{proof}
Identify $\C_v$ with $\C$, write $\cK = \cK_v(\pol)$, $C(\varepsilon) =C_{\pol,v}(\varepsilon)$, and $\kappa(\varepsilon)=\kappa_{\pol,v}(\varepsilon)$. Let $z \in \C \setminus \cK$ and fix $z_0 \in \cK$ with $|z-z_0| = \mathrm{dist}(z,\cK)$. Note that the line segment $L:=[z_0,z]$ is contained in $F_\varepsilon$ since $F$ is convex. Moreover, since $z \notin \cK$, there exists a minimal integer $N$ such that $p^N(L) \not\subseteq F_\varepsilon$. Then, for all $n \leq N$ and all $w \in L$, by the chain rule,
$$
|(\pol^n)'(w)| = |\pol'(\pol^{n-1}(w)) \cdots \pol'(w)| \leq \max \{ |\pol'(z)| : z \in F_\varepsilon \}^n.
$$
On the other hand, there exists $z_1 \in L$ such that $\pol^N(z_1) \notin F_\varepsilon$, and so since $z_0 \in \cK$ gives $\pol^N(z_0) \in \cK \subset F$, we have
\begin{align*}
\varepsilon & < |\pol^N(z_1) - \pol^N(z_0)| = \left| \int_{[z_0,z_1]} (\pol^N)'(w)dw \right| \\
& \leq \int_L |(\pol^N)'(w)|dw \leq \max \{ |\pol'(z)| : z \in F_\varepsilon \}^N \mathrm{dist}(z,\cK),
\end{align*}
and so
$$
\varepsilon^{-\kappa(\varepsilon)}\mathrm{dist}(z,\cK)^{\kappa(\varepsilon)} \geq \left(\max \{ |\pol'(z)| : z \in F_\varepsilon \}^{\kappa(\varepsilon)} \right)^{-N} = d^{-N}.
$$
Thus,
\begin{align} \label{eq:HolJulBd}
\loccanheight{v}{\pol}(z) & = d \loccanheight{v}{\pol}(\pol^{N-1}(z)) d^{-N} \notag \\
& \leq d \varepsilon^{-\kappa(\varepsilon)}\mathrm{dist}(z,\cK)^{\kappa(\varepsilon)} = \frac{C(\varepsilon)}{2} \mathrm{dist}(z,\cK)^{\kappa(\varepsilon)}.
\end{align}
Now, let $z,w \in \C$, and suppose $\mathrm{dist}(z,\cK) \geq \mathrm{dist}(z,\cK)$. If $|z-w| \geq \mathrm{dist}(z,\cK)$, then the result follows from \eqref{eq:HolJulBd}. Otherwise, $w$ lies in the disk centred at $z$ of radius $\mathrm{dist}(z,\cK)$, on which $\loccanheight{v}{\pol}$ is a positive harmonic function. Therefore by Harnack's inequality,
$$
\frac{\mathrm{dist}(z,\cK)-|z-w|}{\mathrm{dist}(z,\cK)+|z-w|} \loccanheight{v}{\pol}(z) \leq \loccanheight{v}{\pol}(w) \leq \frac{\mathrm{dist}(z,\cK)+|z-w|}{\mathrm{dist}(z,\cK)-|z-w|} \loccanheight{v}{\pol}(z).
$$
Rearranging this and applying \eqref{eq:HolJulBd} gives
\begin{align*}
|\loccanheight{v}{\pol}(z)-\loccanheight{v}{\pol}(w)| & \leq \frac{\loccanheight{v}{\pol}(z)+\loccanheight{v}{\pol}(w)}{\mathrm{dist}(z,\cK)} |z-w| \\
& \leq \frac{C(\varepsilon) \mathrm{dist}(z,\cK)^{\kappa(\varepsilon)}}{\mathrm{dist}(z,\cK) |z-w|^{\kappa(\varepsilon)-1}} |z-w|^{\kappa(\varepsilon)} \leq C(\varepsilon) |z-w|^{\kappa(\varepsilon)},
\end{align*}
noting that $|z-w| < \mathrm{dist}(z,\cK)$ and $\kappa(\varepsilon) < 1$.
\end{proof}

In particular, when $\pol = \unipol$ we have the following.

\begin{corollary} \label{cor:ArchHolUni}
Let $v$ be an archimedean place of $K$. Then for all $z,w \in \C_v$,
$$
|\loccanheight{v}{\unipol}(z)-\loccanheight{v}{\unipol}(w)| \leq C_v |z-w|^{\kappa_v},
$$
where $C_v = 4 \log 6 + 2 \log^+ |c|_v$, and
$$
\kappa_v = \frac{2 \log 2}{2 \log 6 + \log^+ |c|_v}.
$$
\end{corollary}

\begin{proof}
Let $F$ be the disk centred at the origin of radius $R := 2 \max \{1, |c|_v^{1/2} \}$. Then it is easy to see that $F \supseteq \cK_v(\unipol)$. Note that $R+1 \leq 3 \max \{ 1, |c|_v^{1/2} \}$, so on $F_1$,
$$
|\unipol'(z)|_v = \left| 2 z \right|_v \leq 6 \max \{ 1, |c|_v^{1/2} \},
$$
and by Lemma~\ref{lem:LocalHeightProps}~(a),
$$
\loccanheight{v}{\unipol}(z) \leq \log 2 + \log \left( 3 \max \{ 1, |c|_v^{1/2} \} \right) .
$$
The result follows from applying Lemma~\ref{lem:HolArch} with $\varepsilon = 1$.
\end{proof}

We now determine H\"{o}lder continuous potentials for the measure $\rho_{r,v}$ associated to a rational function $r \in K(z)$ at a non-archimedean place $v$, following \cite[Proposition~6.5]{FRL} (see also \cite[Theorem~10.35]{BR}).

\begin{lemma} \label{lem:HoldPotGen}
Where $\rat \in K(z)$ has degree $d \geq 2$ and $v$ is a place of bad reduction for $\rat$, there exists $g_v : \BP(\C_v) \to \R$ such that $\Delta g_v = \rho_{\rat,v} - \lambda_v$, and
$$
|g_v(z)-g_v(w)| \leq C_v \ud(z,w)^{\kappa_v}
$$
for all $z,w \in \BP(\C_v)$, where
$$
C_v = \frac{4d^2 \max_{0 \leq k < d} d_{\H_v}(\zeta_G,w_k)}{L_v \min_{0 \leq k < d} \diam_G(w_k)} + d \sum_{k=0}^{d-1} d_{\H_v}(\zeta_G,w_k)
$$
and
$$
\kappa_v = \frac{\log d}{\log L_v}.
$$
Here $L_v := \max \{ 2d, \Lip_v(\rat) \}$, where $\Lip_v(\rat)$ is the Lipschitz constant of $\rat$ as a function on $\BP(\C_v)$, and $w_0,\ldots,w_{d-1} \in \BP(\C_v)$ are the preimages of the Gauss point $\zeta_G$ under $\rat$.
\end{lemma}

\begin{proof}
Fix a place $v$ of bad reduction and write $\lambda_v = \lambda$, $| \cdot |_v = | \cdot |$. Following the proof of \cite[Proposition~6.5]{FRL} (see also \cite[Theorem~10.35]{BR}), if we take a Lipschitz (with respect to $\ud$) potential $g$ (which we will specify later) such that $\Delta g = d^{-1} \rat^* \lambda - \lambda$, then since $\rho_{\rat,v} = \lim_{n \to \infty} d^{-n} (\rat^n)^* \lambda$, we can deduce that $\rho_{\rat,v} = \lambda + \Delta g_v$ where $g_v = \sum_{k=0}^\infty d^{-k} g \circ \rat^k$. Let $N \geq 0$ be an integer. We have for any $z,w \in \BP(\C_v)$,
\begin{align*}
|g_v(z)-g_v(w)| & \leq \sum_{k=0}^{N-1} d^{-k} |g \circ \rat^k(z) - g \circ \rat^k(w)| + \sup |g| \sum_{k=N}^\infty d^{-k} \\
& \leq \left( \sum_{k=0}^{N-1} \frac{\Lip(g) \, \Lip(\rat)^k}{d^k} \right) \ud(z,w) + \sup |g| \sum_{k=N}^\infty d^{-k}.
\end{align*}
Then, writing $E=\Lip(g)$, $L=\max \{2d, \Lip(\rat) \}$ and $B = \sup |g|$, we have
\begin{align*}
|g_v(z)-g_v(w)| & \leq \frac{E}{\frac{L}{d}-1} \left( \frac{L}{d} \right)^N \ud(z,w) + dB \left( \frac{1}{d} \right)^N \\
& = \left( \frac{E}{\frac{L}{d}-1} L^N \ud(z,w) + dB \right) \left( \frac{1}{d} \right)^N.
\end{align*}
If we take $N=\lfloor-\log(\ud(z,w))/\log(L) \rfloor$, then $N = -\log_L \ud(z,w) - \delta$ for some $0 \leq \delta < 1$, and we obtain
\begin{align} \label{eq:HolderBound}
|g_v(z)-g_v(w)| & \leq \left( \frac{E}{\frac{L}{d}-1} L^{-\log_L \ud(z,w)} \ud(z,w) L^{-\delta} + dB \right) d^{\frac{\log \ud(z,w)}{\log L} + \delta} \notag \\
& = d^{\delta} \left( \frac{dE}{L-d} L^{-\delta} + dB \right) \ud(z,w)^{\frac{\log d}{\log L}} \notag \\ 
& < \left(\frac{2d^2 E}{L}+d^2B \right) \ud(z,w)^{\frac{\log d}{\log L}},
\end{align}
noting that $L \geq 2d$.

Now let $w_0,\ldots,w_{d-1}$ be the preimages of the Gauss point $\zeta_G$ under $\rat$. With $\langle \cdot, \cdot \rangle_G$ defined as in \eqref{eq:GromovDef}, let
$$
H_k(z) = \langle z, w_k \rangle_G.
$$
That is, referring to \eqref{eq:GromovDef},
$$
H_k(z) = d_{\H_v}(m_z, \zeta_G),
$$
where $d_{\H_v}$ is the path distance defined in \eqref{eq:defPathDistance}, and $m_z = z \vee_G w_k$ denotes the unique point of $[z,w_k] \cap [z, \zeta_G] \cap [w_k, \zeta_G]$. Clearly $H_k$ is locally constant off the path $[\zeta_G, w_k]$, and so
\begin{equation} \label{eq:supHk}
0 = H_k(\zeta_G) \leq H_k(z) \leq H_k(w_k) = d_{\H_v}(\zeta_G, w_k).
\end{equation}

We have from \eqref{eq:LaplacGromov}
$$
\Delta H_k = \lambda - [w_k],
$$
so if $g = - \frac{1}{d} \sum_{0 \leq k < d} H_k$, then $\Delta g = d^{-1} \rat^* \lambda - \lambda$, as desired. Note first that from \eqref{eq:supHk}, $g$ is bounded with
$$
B = \sup |g| \leq \frac{1}{d} \sum_{k=0}^{d-1} d_{\H_v} (\zeta_G, w_k). 
$$
Moreover, $g$ is Lipschitz with constant 
$$
\Lip(g) = E \leq  \frac{2 \max_{0 \leq k < d} d_{\H_v}(\zeta_G,w_k)}{\min_{0 \leq k < d} \diam_G(w_k)},
$$ 
where $\diam_G$ is defined in \eqref{eq:defGaussDiam}.

Indeed, it suffices to show that each $H_k$ is Lipschitz with constant $E$ for points $x,y \in \BP(\C_v)$ with $y \in [x,\zeta_G]$, since given arbitrary $x,y  \in \BP(\C_v)$, putting $z = x \vee_G y$ gives
\begin{align*}
|H_k(x)-H_k(y)| & \leq |H_k(x)-H_k(z)|+|H_k(z)-H_k(y)| \\
& \leq E \ud(x,z)+ E \ud(z,y) \leq E \ud(x,y).
\end{align*}
Take $x,y \in \BP(\C_v)$ with $y \in [x,\zeta_G]$. Writing $[x,\zeta_G] = [x,m_x) \cup [m_x,\zeta_G]$, and noting that $[m_x,\zeta_G] \subseteq [w_k,\zeta_G]$ we see that if $y \in [x,m_x)$, then $m_y = m_x$ and so $H_k(x)=H_k(y)$, and otherwise $m_y = y$. Suppose we are in the latter case. We have
\begin{align*}
|H_k(x) - H_k(y)| & = |d_{\H_v}(m_x,\zeta_G) - d_{\H_v}(y,\zeta_G)| \\
& = d_{\H_v}(m_x,y) \leq d_{\H_v}(w_k,\zeta_G),
\end{align*}
and moreover $x \vee_G y = y$, whence
$$
\ud(x,y) = \diam_G(y)-\diam_G(x).
$$
Since $y \in [\zeta_G,w_k]$, $\diam_G(y) \geq \diam_G(w_k)$. 

Suppose that
$$
\diam_G(x) \leq \frac{1}{2} \diam_G(w_k).
$$
Then
$$
\ud(x,y) \geq \frac{1}{2}\diam_G(w_k),
$$
and so
\begin{align*}
|H_k(x)-H_k(y)| & \leq d_{\H_v}(w_k,\zeta_G) \\
& = \frac{2 d_{\H_v}(w_k,\zeta_G)}{\diam_G(w_k)} \cdot \frac{1}{2} \diam_G(w_k) \leq E \ud(x,y).
\end{align*}
On the other hand, if $\diam_G(x) \geq \diam_G(w_k)/2$,
then, since $\log$ is Lipschitz on $[\diam_G(w_k)/2,1]$ with constant $2/\diam_G(w_k)$, we have
\begin{align*}
|H_k(x)-H_k(y)| = d_{\H_v}(m_x,y) & \leq d_{\H_v}(x,y) \\
& = |d_{\H_v}(\zeta_G,x)-d_{\H_v}(\zeta_G,y)| \\
& = |\log \diam_G(y) - \log \diam_G(x)|\\
& \leq \frac{2}{\diam_G(w_k)} |\diam_G(y)-\diam_G(x)| \\
& = \frac{2}{\diam_G(w_k)} \ud(x,y) \leq E \ud(x,y),
\end{align*}
as desired.

Substituting these values for $E$ and $B$ into \eqref{eq:HolderBound} gives
\begin{align*}
|g_v(z)-g_v(w)| \leq \left( \frac{4d^2 \max_{0 \leq k < d} d_{\H_v}(\zeta_G,w_k)}{L \min_{0 \leq k < d} \diam_G(w_k)} + d \sum_{k=0}^{d-1} d_{\H_v}(\zeta_G,w_k) \right) \ud(z,w)^{\frac{\log d}{\log L}}
\end{align*}
completing the proof.
\end{proof}

The action of a unicritical polynomial $\unipol$ on Berkovich points is easy to compute using results found in \cite{B}. We can hence explicitly bound the quantities in Lemma~\ref{lem:HoldPotGen} and obtain the following.

\begin{lemma} \label{lem:HoldPot}
Where $\unipol(z)=z^2+c$, $\rho_{\unipol}$ has H\"{o}lder continuous potentials with exponents
\begin{equation} \label{eq:HolderExponent}
\kappa_v \geq \begin{cases} \frac{2 \log 2}{2 \log 6 + \log^+|c|_v} & v \in M_K^\infty \\
1 & v \in M^{0,\unipol}_{K,\mathrm{good}} \\
\frac{\log 2}{\log 4 + 4 \log |c|_v} & v \in M^{0,\unipol}_{K,\mathrm{bad}}
\end{cases}
\end{equation}
and constants
\begin{equation} \label{eq:HolderConstant}
C_v \leq \begin{cases} 4 \log 6 + 2 \log^+ |c|_v & v \in M_K^\infty \\ 0 & v \in M^{0,\unipol}_{K,\mathrm{good}} \\ 1 + 6 \log |c|_v & v \in M^{0,\unipol}_{K,\mathrm{bad}}. \end{cases}
\end{equation}

\end{lemma}

\begin{proof}
The archimedean case follows immediately from Corollary~\ref{cor:ArchHolUni}, noting that $\rho_{\unipol,v} = \Delta \loccanheight{v}{\unipol}$ from \eqref{eq:ArchPotLCH}.

Now suppose $v$ is a finite place of bad reduction, so we can apply Lemma~\ref{lem:HoldPotGen}. From \cite[Theorem~0.1]{RW},
$$
\Lip(\unipol) \leq \max \left \{ \frac{2}{|\mathrm{Res}(\unipol)|}, \frac{1}{|\mathrm{Res}(\unipol)|^2} \right \} <  4 |c|^{4}.
$$
We now compute the preimages $w_0,w_1$ of the Gauss point $\zeta_G$ under $\unipol$. By \cite[Proposition~7.6]{B}, if $\zeta(a,r)$ is a point of Type II or III, and $\unipol(D(a,r)) = D(b,s)$, then $\unipol(\zeta(a,r)) = \zeta(b,s)$. Moreover, from \cite[Theorem~3.15]{B}, if $\unipol(z) = \sum_{n \geq 0} c_n (z-a)^n$, then $\unipol(D(a,r)) = D(c_0,t)$, where $t = \max_{i \geq 1} |c_i| r^i$. For $a \in \C_v$, we can write
$$
\unipol(z) = z^2 + c = a^2+c + 2a(z-a)+(z-a)^2,
$$
and so
$$
\unipol(\zeta(a,r)) = \zeta \left( a^2+c, \max\{ |2a|r, r^2 \} \right).
$$
From this we can see that
$$
w_k = \zeta \left( (-1)^k (-c)^{\frac{1}{2}}, \min \left \{ 1, \frac{1}{|2||c|^{\frac{1}{2}}} \right \} \right), \qquad k=0,1,
$$
are the preimages of the Gauss point under $f_c$. Now, we have (see for example \cite[Exercise~6.22]{B})
\begin{equation*}
\diam( \zeta_G \vee w_k ) = \max \{ 1, \| T \|_{w_k} \} = |c|^{1/2},
\end{equation*}
noting that $|c| > 1$.
Thus
$$
d_{\H_v}(\zeta_G,w_k) = 2 \log |c|^{1/2} - \log \min \left \{ 1, \frac{1}{|2||c|^{\frac{1}{2}}} \right \}  \leq \frac{3}{2} \log |c|,
$$
and so
$$
\diam_G(w_k) \geq \frac{1}{|c|^{\frac{3}{2}}}.
$$
Substituting these values into the equations for $C_v$ and $\kappa_v$ in Lemma~\ref{lem:HoldPotGen}, and noting again that $|c| > 1$, completes the proof.
\end{proof}




\section{Truncation function} \label{sec:Trunc}

In this section, for each place $v$ of $K$ we define a suitable test function with which to apply Theorem~\ref{thm:QuantEquid}, and bound its Lipschitz constant and Dirichlet form.

\subsection{Archimedean case} Let $v$ be an archimedean place of $K$ and identify $\C_v$ with $\C$. Let $0 < \delta < 1$ and consider the  function $\phi_\delta(z)$ given as the extension to $\P^1(\C)$ of $\log \max \{ \delta, |z-\alpha| \}$. We have that $\phi_\delta$ is Lipschitz continuous on $\C$ (with respect to the Euclidean distance) with Lipschitz constant
\begin{equation} \label{eq:ArchLipOfTrunc}
\mathrm{Lip}(\phi_\delta) = \frac{1}{\delta}.
\end{equation}

On a chart $U_i$ with coordinates $x,y$, where $u_i(\alpha)=(a_i,b_i)$ we have
$$
\phi_\delta \circ \alpha_i^{-1}(x,y) = \log \max \{ \delta, \sqrt{(x-a_i)^2+(y-b_i)^2} \},
$$
which has weak partial derivatives
\begin{align*}
\frac{\partial \phi_\delta}{\partial x}(x,y) & = \I \left \{\sqrt{(x-a_i)^2+(y-b_i)^2} > \delta \right \} \frac{x-a_i}{(x-a_i)^2+(y-b_i)^2}, \\
\frac{\partial \phi_\delta}{\partial y}(x,y) & = \I \left \{\sqrt{(x-a_i)^2+(y-b_i)^2} > \delta \right \} \frac{y-b_i}{(x-a_i)^2+(y-b_i)^2},
\end{align*}
where $\I$ is an indicator function. 

Hence from \eqref{eq:DirichletForm}, we have (making a substitution to move $(a_i,b_i)$ to the origin)
\begin{align} \label{eq:ArchDirFormOfTrunc}
\langle \phi_\delta, \phi_\delta \rangle & = \int_{\overline D(0,1) \setminus \overline D(0,\delta)}  \frac{dx dy}{x^2+y^2} + \int_{D(0,1) \setminus \overline D(0,\delta)} \frac{dx dy}{x^2+y^2} \notag \\
& = 2 \int_0^{2 \pi} \int_\delta^1 \frac{dr d \theta}{r} = -4 \pi \log \delta.
\end{align}

\subsection{Non-archimedean case} When $v$ is a finite place of $K$, for $0 < \delta < 1$ we define $\phi_\delta : \BP(\C_v) \to \R$ by
$$
\phi_\delta(x) =  \log \max \{ \diam(x \vee \alpha), \delta \},
$$
which extends the function $\log \max \{ |x-\alpha|, \delta \}$ on $\P^1(\C_v)$ to $\BP(\C_v)$. Let $\xi := \zeta(\alpha, 1)$ and $\Lambda := [\zeta(\alpha,\delta), \xi]$. Then $\phi_\delta$ is locally constant outside of $\Lambda$. Indeed, it is easy to see that $\phi_\delta(x) = \phi_\delta(x \vee_\xi \zeta(\alpha,\delta))$. Moreover, since for $x \in \Lambda$ we have
$$
\phi_\delta(x) = \log \diam(x),
$$ 
$\phi_\delta$ is $\cC^1$ on $\Lambda$, with
$$
\partial \phi_\delta(x) = \lim_{\substack{y \to x \\ y \in [\xi,x]}} \frac{\phi_\delta(x)-\phi_\delta(y)}{d_{\H_v}(x,y)} = 1.
$$
Note that this implies $\phi_\delta$ is Lipschitz with 
\begin{equation} \label{eq:nonArchLipofTrunc}
\Lip(\phi_\delta) = 1. 
\end{equation}
We conclude that
\begin{align} \label{eq:nonArchDirFormOfTrunc}
\langle \phi_\delta, \phi_\delta \rangle & = \int_{\BP(\C_v)} (\partial \phi_\delta)^2 d\mu  = \int_\Lambda d\mu \notag = \mu(\Lambda) \\
& = d_{\H_v}(\zeta(\alpha,\delta),\xi) = - \log \delta.
\end{align}

Note that for all places $v$, if $\delta \leq \delta_v(\alpha)$, then $\phi_\delta$ agrees with $\log | \cdot - \alpha|$ at all preperiodic points of $\unipol$, and on the support of $\rho_{\unipol,v}$.

\section{Proof of Theorem~\ref{thm:main}} \label{sec:MainProof}

We can now prove the following explicit form of Theorem~\ref{thm:main}.

\begin{theorem} \label{thm:main2}
Let $K$ be a number field, and let $\pol \in K[z]$ be a polynomial of degree $d \geq 2$. Let $S$ be a finite set of places of $K$ containing all the archimedean ones, as well as all the places of bad reduction for $\pol$. Suppose that $\alpha \in \overline K$ is not preperiodic under $\pol$, and further that $\alpha$ is totally Fatou at all places $v \in M_K$. Let $C_v$ and $\kappa_v$ be constants as in \eqref{eq:defHCP} for the adelic measure associated to $\pol$, and let $C \geq 1$ and $\kappa > 0$ be constants satisfying $C \geq \sum_{v \in V} C_v$ and $\kappa \leq \min_{v \in V} \kappa_v$, where $V = M_K^\infty \cup M_{K,\mathrm{bad}}^{0,\pol}$. For each $v \in S$, write $\delta_v = \delta_{\pol,v}(\alpha)$. Then
\begin{align*}
|\relsupreper{\pol}{S}{\alpha}| \ll \max \Bigg \{ C, & \: \: \left( \frac{|V| \sum_{v \in S} |\log \delta_v|^{1/2}}{\kappa \hat h(\alpha)^2} \right)^2, \\
& \qquad \frac{|V|}{\kappa} \left( \frac{2}{\hat h(\alpha)} \left( \sum_{v \in M_K^\infty} \frac{1}{\delta_v} + |S \setminus M_K^\infty| \right) \right)^{\kappa} \Bigg \}.
\end{align*}
\end{theorem}

\begin{proof}
First note that $\cP := \relsupreper{\pol}{S}{\alpha}$ is $\Gal(\overline K/K)$-invariant, and write
$$
\cP = \bigcup_{\beta \in Y} \Gal(\overline K/K) \cdot \beta
$$
as a union of Galois orbits. Define, for each $v \in M_K$,
\begin{align*}
\Gamma_v & := \frac{1}{|\cP|} \sum_{z \in \cP} \log |z-\alpha|_v \\
& = \frac{1}{|\cP|} \sum_{\beta \in Y} \sum_{\sigma : K(\beta)/K \hookrightarrow \overline K_v} \log |\sigma(\beta)-\alpha|_v,
\end{align*}
and set $\Gamma = \sum_{v \in M_K} \Gamma_v$. First consider a place $v \notin S$. By definition such a place is of good reduction, and so $\loccanheight{v}{\pol}(\alpha) = \log^+|\alpha|_v$. Note that as the points $z \in \cP$ are preperiodic, we have $0 = \loccanheight{v}{\pol}(z) = \log^+ |z|_v$, and so $|z|_v \leq 1$. Thus, if $|\alpha|_v > 1$, $\log |z-\alpha| = \log^+ |\alpha| = \loccanheight{v}{\pol}(\alpha)$ for all $z \in \cP$. If $|\alpha|_v \leq 1$, then for $z \in \cP$, $|z-\alpha|_v \leq \max \{ |z|_v, |\alpha|_v \} \leq 1$, and moreover the integrality hypothesis gives $|z-\alpha| \geq 1$. Hence again we have $\log |z-\alpha| = \loccanheight{v}{\pol}(\alpha)$, and so
$$
\Gamma_v = \loccanheight{v}{\pol}(\alpha)
$$
Therefore
\begin{equation} \label{eq:GamForm}
\Gamma = \canheight{\pol}(\alpha) + \sum_{v \in S} \left( \Gamma_v - \loccanheight{v}{\pol}(\alpha) \right).
\end{equation}
Since we have shown that $\Gamma_v=0$ for all but finitely many places, we may exchange sums in the following to obtain
\begin{align} \label{eq:GamZero}
\Gamma = \sum_{v \in M_K} \Gamma_v & = \frac{1}{|\cP|} \sum_{v \in M_K} \sum_{\beta \in Y} \sum_{\sigma : K(\beta)/K \hookrightarrow \overline K_v} \log |\sigma(\beta)-\alpha|_v \notag \\
& = \frac{1}{|\cP|} \sum_{\beta \in Y} \sum_{v \in M_K} \sum_{\sigma : K(\beta)/K \hookrightarrow \overline K_v} \log |\sigma(\beta)-\alpha|_v \notag \\
& = \frac{1}{|\cP|} \sum_{\beta \in Y} \sum_{w \in M_{K(\beta)}} \log |\beta-\alpha|_w = 0,
\end{align}
where the last equality follows from the product formula, noting that each $\beta \neq \alpha$, as $\alpha$ is not preperiodic for $\pol$ by assumption. Suppose $|\cP| \geq (6C\kappa)/(|V|+1)$. As $\cP$ consists of preperiodic points for $\pol$, $h_{\rho_{\pol}}(\cP) = 0$, so for each $v \in S$, applying Theorem~\ref{thm:QuantEquid} and \eqref{eq:MahlerFormula} to the truncation function $\phi_{\delta_v}$ (see \textsection \ref{sec:Trunc}) gives
$$
|\Gamma_v - \loccanheight{v}{\pol}(\alpha) | \leq \left( \frac{2(|V|+1)}{\kappa} \frac{\log |\cP|}{|\cP|} \langle \phi_{\delta_v}, \phi_{\delta_v} \rangle \right)^{1/2} + \Lip(\phi_{\delta_v}) \left( \frac{|V|+1}{2C\kappa |\cP|} \right)^{1/\kappa}.
$$
Plugging this into \eqref{eq:GamForm}, from \eqref{eq:ArchLipOfTrunc}, \eqref{eq:ArchDirFormOfTrunc}, \eqref{eq:nonArchLipofTrunc} and \eqref{eq:nonArchDirFormOfTrunc} we have
\begin{align*}
\Gamma & \geq \canheight{\pol}(\alpha) - \sum_{v \in M_K^\infty} \left[ \left( 8\pi \frac{|V|+1}{\kappa} \frac{\log |\cP|}{|\cP|} |\log \delta_v| \right)^{1/2} + \frac{1}{\delta_v} \left( \frac{|V|+1}{2C\kappa |\cP|} \right)^{1/\kappa} \right] \\
& \qquad \qquad \qquad - \sum_{v \in S \setminus M_K^\infty} \left[ \left( \frac{|V|+1}{\kappa} \frac{\log |\cP|}{|\cP|} |\log \delta_v| \right)^{1/2} + \left( \frac{|V|+1}{2C\kappa |\cP|} \right)^{1/\kappa} \right] \\
& \geq \canheight{\pol}(\alpha) - \left(8\pi \frac{|V|+1}{\kappa} \frac{\log |\cP|}{|\cP|} \right)^{1/2} \left( \sum_{v \in S} |\log \delta_v|^{1/2} \right) \\
& \qquad \qquad - \left[ \left( \sum_{v \in M_K^\infty} \frac{1}{\delta_v} \right) + |S \setminus M_K^\infty| \right] \left( \frac{|V|+1}{2C\kappa |\cP|} \right)^{1/\kappa}.
\end{align*}
Recall that the Lambert function (see \cite{C}), denoted $W(z)$ for $z \in \C$, is defined as the function that satisfies $W(z) e^{W(z)} = z$. If $z \in \R$, then $W(z)$ can take two possible real values for $-1/e \leq z \leq 0$. Let $W_{-1}$ denote the branch whose values satisfy $W(z) \leq -1$. For $-1/e \leq z \leq 0$ we have
$$
\frac{\log e^{-W_{-1}(z)}}{e^{-W_{-1}(z)}} = \frac{-W_{-1}(z)}{W_{-1}(z)/z} = -z,
$$
and so if $|\cP| \geq e^{-W_{-1}(z)}$, then $(\log |\cP|)/|\cP| \leq -z$.
By \cite[Theorem~1]{C},
$$
W_{-1}(z) > -1- \sqrt{2u}- u,
$$
where $u = -\log(-z)-1$. Thus, if we suppose
\begin{equation} \label{eq:PBound}
|\cP| > \max \left \{ \frac{6C\kappa}{|V|+1}, e^{1+\sqrt{2u}+u}, \frac{|V|+1}{2C\kappa} \left[ \frac{2}{\canheight{\pol}(\alpha)} \left( \left( \sum_{v \in M_K^\infty} \frac{1}{\delta_v} \right) + |S \setminus M_K^\infty | \right) \right]^\kappa \right \},
\end{equation}
where
$$
u = - \log \min \left \{ \frac{1}{e}, \frac{\kappa \canheight{\pol}(\alpha)}{32 \pi (|V|+1) \left( \sum_{v \in S} |\log \delta_v |^{1/2} \right)^2} \right \} - 1,
$$
then
$$
\Gamma > \canheight{\pol}(\alpha) - \frac{\canheight{\pol}(\alpha)}{2} - \frac{\canheight{\pol}(\alpha)}{2} = 0,
$$
which contradicts \eqref{eq:GamZero}. The result follows from noting that $C \geq 1$, $\kappa \leq 1$ and $e^{1+\sqrt{2u}+u} \ll e^{2u}$.
\end{proof}

\section{Bounds on the size of preperiodic points at archimedean places} \label{sec:arch}

Let $v$ be an archimedean place of $K$, identify $\C_v$ with $\C$, writing $| \cdot |$ for $| \cdot |_v$, and $\cJ(\unipol)$ and $\cK(\unipol)$ respectively for the $v$-adic Julia set and filled Julia set of $\unipol$.

\subsection{Parameters outside the Mandelbrot set}

Suppose $c$ lies outside the $v$-adic Mandelbrot set. Then the Julia set $\cJ(\unipol)$ of $\unipol$ is a Cantor set which coincides with the filled Julia set $\cK(\unipol)$. As such, all preperiodic points for $\unipol$ lie in $\cJ(\unipol)$, and so the quantity $\delta_v = \delta_{\unipol,v}(0)$ defined in \eqref{eq:defdeltav} is precisely the distance $\mathrm{dist}(0,\cJ(\unipol))$ from 0 to the Julia set. 

We call $R$ an \emph{escape radius} for $\unipol(z)=z^d+c$ if for all $z \in \C$ with $|z| > R$, $|\unipol^n(z)| \to \infty$. A trivial escape radius is $R=1+|c|$, but there are better estimates. Namely, we have the following \cite[Corollary~3.3]{S}.

\begin{lemma}
Define
$$
Q_c(R) = R^2-R-|c|,
$$
and let $R_c$ be the largest real fixed point of $Q_c$. Then 
$$
R_c = \frac{1}{2}+\sqrt{\frac{1}{4}+|c|}
$$
is an escape radius for $\unipol$.
\end{lemma}

Note that $|c| > R_c$ whenever $|c| > 2$.

When $|c|$ is sufficiently large, it is easy to show that for a point $z$ close to 0, $\unipol(z)$ lies outside an escape radius of $\unipol$, so $z$ cannot be preperiodic for $\unipol$, and hence does not lie inside the Julia set. In particular we have the following.

\begin{lemma} \label{lem:JuliaSetDistance}
Suppose $|c| > 2$. Then the distance from 0 to $\cJ(\unipol)$
$$
\delta_v = \mathrm{dist}(0,\cJ(\unipol)) \geq \left( |c|-R_c \right)^{\frac{1}{2}}.
$$
We also have
$$
\delta_v = \mathrm{dist}(0, \cJ(\unipol)) \geq \frac{1}{2}
$$
for $c \in [1,\infty)$. In particular, this means that for all $c \in \Z$ such that $0$ is not preperiodic for $\unipol$, we have
$$
- \log \delta_v \leq \log 2.
$$
\end{lemma}

\begin{proof}
Suppose $z \in \C$ with $|z| < (|c|-R_c)^{\frac{1}{2}}$. Then
$$
|\unipol(z)| = |z^2+c| \geq |c|-|z|^2 > R_c,
$$
whence $z \notin \cJ(\unipol)$.

Now, suppose $c \in [1, \infty)$ note that the image of $\D(0,1/2)$ under $\unipol$ is $\D(c,1/2^d)$, so for any $z \in \D(0,1/2)$, $|\mathrm{Arg}(\unipol(z))| \leq \tan^{-1}(1/4)$. Hence $|\mathrm{Arg}(\unipol(z))^2| \leq 2 \tan^{-1}(4^{-1}) < \pi/4$, and so 
$$
|\unipol^2(z)| > |c| + \frac{|\unipol(z)|^2}{\sqrt{2}} \geq |c| + \frac{1-4^{-1}}{\sqrt{2}}.
$$
It is easy then to check (just applying the reverse triangle inequality) that certainly $\unipol^4(z) > R_{d,c}$, and so $z \notin \cJ(\unipol)$. 
\end{proof}



On the other hand, when $|c|$ is not large, and especially when $c$ lies close to the boundary of the Mandelbrot set, we cannot argue as above, as it could take many iterations under $\unipol$ for a point close to 0 to get outside an escape radius. This information is contained, however, in the local canonical height of 0 under $\unipol$. Indeed, we can still obtain a lower bound for $\mathrm{dist}(0,\cJ(\unipol))$ in terms of $\loccanheight{v}{\unipol}(0)$ by using the H\"{o}lder continuity of $\loccanheight{v}{\unipol}$.

\begin{prop} \label{prop:Kosek}
Let $v \in M_K^\infty$ and suppose $c \in K$ lies outside the $v$-adic Mandelbrot set. Then we have
\begin{align*}
\delta_v & = \mathrm{dist}(0, \cJ(\unipol)) \\
& \geq \left( \frac{\loccanheight{v}{\unipol}(0)}{2 \log 6 + \log^+ |c|_v} \right)^{\frac{2 \log 6 + \log^+|c|_v}{2 \log 2}}.
\end{align*}
\end{prop}

\begin{proof}
This follows immediately from rearranging \eqref{eq:HolJulBd} with $z=0$ and the constants from Corollary~\ref{cor:ArchHolUni}.
\end{proof}

We now reinterpret Lemma~\ref{lem:JuliaSetDistance} and Proposition~\ref{prop:Kosek} in order to give a bound of a form amenable to the proof of Theorem~\ref{thm:UniformMain}.

\begin{corollary} \label{cor:ArchOutParamDeltavBound}
Let $\varepsilon > 0$, $v \in M_K^\infty$ and let $c \in K$ be such that $\loccanheight{v}{\unipol}(0) \geq \varepsilon$. Then
$$
-\log \delta_v \leq A_{\infty,1},
$$
where
$$
A_{\infty,1} = \begin{cases} \max \left \{0, -\frac{1}{2} \log \left(e^{2\left(\varepsilon-\log 2 \right)}- R_{e^{2\left(\varepsilon-\log 2 \right)}} \right) \right\} & \varepsilon > \frac{3 \log 2}{2}, \\
\frac{\log 12 + \varepsilon}{\log 2} \log \left( \frac{\log 48 + 2 \varepsilon}{\varepsilon} \right) & \varepsilon > 0. \end{cases}
$$
Note that we separate the case $\varepsilon > \frac{3 \log 2}{2}$ because it ensures the parameter $c$ satisfies $|c| > 2$, allowing us to apply Lemma~\ref{lem:JuliaSetDistance}. This produces a better bound than the latter, general case, for which we remark that $A_{\infty, 1} \ll -\log \varepsilon$ as $\varepsilon \to 0$.
\end{corollary}
 
\begin{proof}
Suppose $\varepsilon > \frac{3 \log 2}{2}$. We must have $|c| \geq 1$ as if $|c| < 1$, an easy induction gives
$$
|\unipol^n(0)| < 2^{2^{n-1}-1},
$$
and so $\loccanheight{v}{\unipol}(0) < (\log 2)/2$. Hence, by Lemma~\ref{lem:LocalHeightProps}~(a), we have
$$
\loccanheight{v}{\unipol}(0) \leq \log 2 + \frac{1}{2} \log |c|,
$$
and so
$$
|c| > e^{2\left(\loccanheight{v}{\unipol}(0)-\log 2 \right)} = 2.
$$
Hence, by Lemma~\ref{lem:JuliaSetDistance},
$$
\delta_v \geq (|c| - R_c)^{1/2}.
$$
Note that $|c|-R_c$ increases with $|c|$ for $|c| \geq 1$. Thus we conclude this case by substituting $|c| > e^{2\left(\varepsilon-\log 2\right)}$.

Now, for arbitrary $\varepsilon$, from Lemma~\ref{lem:LocalHeightProps}~(b), we have
$$
2 \loccanheight{v}{\unipol}(0) = \loccanheight{v}{\unipol}(c) = \max \{ \loccanheight{v}{\unipol}(c), \loccanheight{v}{\unipol}(0) \} \geq \log |c| - \log 4,
$$
and so
\begin{equation} \label{eq:ska1}
\log^+ |c| \leq \log 4 + 2\varepsilon.
\end{equation}
The result follows from plugging \eqref{eq:ska1} into Proposition~\ref{prop:Kosek} and taking logarithms.
\end{proof}

\subsection{Parameters inside the hyperbolic locus}

Now suppose $c$ lies in a period $t$ hyperbolic component of the $v$-adic Mandelbrot set. That is, $\unipol$ has an attracting cycle of period $t$. Let $U$ be the component of the immediate basin of this cycle containing 0, let $\alpha$ be the element of the cycle in this component, and let $\phi : U \to \D$ be the uniformizing map with $\phi(\alpha)=0$, appropriately normalized so that $g:= \phi \circ \unipol^t \circ \phi^{-1}$ is the Blaschke product
$$
g(z) = z \frac{z+\lambda}{1+\bar \lambda z},
$$ 
where $\lambda = (\unipol^t)'(\alpha) \neq 0$ is the multiplier of $\alpha$ (see \cite[\textsection \textsection 25.3]{L}).

Note that $\phi(0)$ is a critical point of $g$ lying in $\D$, and so
\begin{equation} \label{eq:phi}
\phi(0) = \frac{1-\sqrt{1-|\lambda|^2}}{\bar \lambda} = \frac{1-\sqrt{1-|\lambda|^2}}{|\lambda|^2} \lambda.
\end{equation}
Suppose that $z \in \phi^{-1}(\D(\phi(0),\delta)) \subset U$ is preperiodic for $\unipol$, where
\begin{equation} \label{eq:delta}
\delta := |\lambda - \phi(0)| = \frac{|\lambda|^2-1+\sqrt{1-|\lambda|^2}}{|\lambda|}.
\end{equation}
Then $\unipol^{tk}(z)=\alpha$ for some $k \geq 1$, and so $g^k(\phi(z)) = \phi(\unipol^{tk}(z)) = \phi(\alpha)=0$. By the Schwarz lemma, $|g(w)| < |w|$ for all $w \in \D \setminus \{0 \}$, and so, since the only preimages of 0 under $g$ are $0$ and $-\lambda$, the disk centred at 0 of radius $|\lambda|$ contains no iterated preimages of $0$ under $g$ other than $0$ itself. But we have
\begin{align*}
|\phi(z)| & = |\phi(z)-\phi(0)+\phi(0)| \leq |\phi(z)-\phi(0)| + |\phi(0)| \\
& < \delta + |\phi(0)| = |\lambda|,
\end{align*}
a contradiction. Hence $\phi^{-1} (\D(\phi(0),\delta)) \subset U$ contains no preperiodic points under $\unipol$.

Let $\eta : \D \to \D$ be given by $\eta(z) = \phi(d(\alpha,\partial U)z+\alpha)$. Then $\eta(0)=0$, so by the Schwarz lemma,
$$
1 \geq |\eta'(0)| = |d(\alpha, \partial U) \phi'(\alpha)|,
$$
and hence by the inverse function theroem,
\begin{equation} \label{eq:SchwazDistBound}
\left|(\phi^{-1})'(0) \right| \geq d(\alpha, \partial U).
\end{equation}

Recall the Koebe distortion theorem \cite[\textsection 1.2]{Po}.

\begin{theorem}
Let $F : \D \to \C$ be univalent with $F(0)=0$ and $F'(0)=1$. Then for all $z \in \D$,
$$
\frac{|z|}{(1+|z|)^2} \leq |F(z)| \leq \frac{|z|}{(1-|z|)^2}
$$
and
$$
\frac{1-|z|}{(1+|z|)^3} \leq |F'(z)| \leq \frac{1+|z|}{(1-|z|)^3}.
$$
\end{theorem}

Applying the Koebe distortion theorem to
$$
F(z):=\frac{\phi^{-1}(z)-\alpha}{(\phi^{-1})'(0)}
$$
at $z=\phi(0)$, we get
\begin{equation} \label{eq:KobeDist}
\left| (\phi^{-1})'(\phi(0)) \right| = \left| (\phi^{-1})'(0) F'(\phi(0)) \right| \geq \frac{1-|\phi(0)|}{(1+|\phi(0)|)^3} \left| (\phi^{-1})'(0) \right|.
\end{equation}
Let $\psi : \D \to \D(\phi(0),\delta)$ be given by
$$
\psi(z) = \delta z + \phi(0)
$$
so that by the Koebe 1/4-theorem, $\phi^{-1}(\D(\phi(0),\delta))$ contains a disk about $\phi^{-1} \circ \psi(0) = 0$ of radius
\begin{align} \label{eq:boundo}
\left| \frac{(\phi^{-1} \circ \psi)'(0)}{4} \right| & = \frac{\delta}{4} \left| (\phi^{-1})'(\phi(0)) \right| \notag \\
& \geq \frac{\delta(1-|\phi(0)|)}{4(1+|\phi(0)|)^3} \left| (\phi^{-1})'(0) \right|  \notag \\
& = \frac{|\lambda| \left( |\lambda|^2 - 1 + \sqrt{1-|\lambda|^2} \right) \left( |\lambda| - 1 + \sqrt{1- \lambda|^2} \right)}{4 \left( |\lambda| + 1 - \sqrt{1-|\lambda|^2} \right)^3} \left| (\phi^{-1})'(0) \right| \notag \\
& \geq \frac{|\lambda| \left( |\lambda|^2 - 1 + \sqrt{1-|\lambda|^2} \right) \left( |\lambda| - 1 + \sqrt{1- |\lambda|^2} \right)}{4 \left( |\lambda| + 1 - \sqrt{1-|\lambda|^2} \right)^3} d(\alpha, \partial U),
\end{align}
where the first inequality comes from \eqref{eq:KobeDist}, the next equality results from substituting \eqref{eq:delta} and \eqref{eq:phi}, and the last inequality follows from \eqref{eq:SchwazDistBound}. 

Now, the quantity $d(\alpha, \partial U)$ is known as the \emph{converge radius} for the periodic point $\alpha$. In \cite{S}, Stroh obtains some lower bounds for converge radii. In particular the following will be useful \cite[Corollary~3.14]{S}.

\begin{lemma}
Let $P(z) = \sum_{i=0}^d a_i z^i$, and let $\tilde z$ be an attracting fixed point of $P$. Furthermore, let $\tilde P(z) = \sum_{i=0}^d b_i z^i = P(z+\tilde z)-\tilde z$ and
$$
\tilde Q(r) = \left( \sum_{i=1}^d |b_i| r^{i-1} \right) - 1.
$$
Then $\tilde r_{\tilde Q}$ defined by
$$
\tilde r_{\tilde Q} = \min \{ r \in \R^+ : \tilde Q(r) = 0 \} 
$$
satisfies $D(\tilde z, \tilde r_{\tilde Q}) \subset A^*(\tilde z)$ (the immediate basin of $\tilde z$).
\end{lemma}

In our case, with $P = \unipol^t$ and $\tilde z = \alpha$, so that $\tilde P(z) = \unipol^t(z+\alpha)-\alpha$, we can calculate
$$
\tilde Q(r) = r^{2^t-1} + \frac{(\unipol^t)^{(2^t-1)}(\alpha)}{(2^t-1)!} r^{2^t-2} + \cdots + \frac{(\unipol^t)''(\alpha)}{2!}r + |\lambda| - 1.
$$
Note that if $t=1$, we have 
\begin{equation} \label{eq:ConvRad1}
d(\alpha, \partial U) \geq \tilde r_{\tilde Q} = 1-|\lambda|.
\end{equation}
Note that $|c| \leq 2$, since $c$ lies in the Mandelbrot set. Also, since $\alpha$ is periodic, $|\alpha| \leq R_c \leq 2$. 

For $t \geq 2$, the roots of
$$
r^{2^t-1} \tilde Q \left( \frac{1}{r} \right) = (|\lambda| - 1) r^{2^t-1} + \frac{(\unipol^t)''(\alpha)}{2!} r^{2^t-2} + \cdots + \frac{(\unipol^t)^{(2^t-1)}(\alpha)}{(2^t-1)!} r + 1
$$
have modulus at most
$$
1 + \frac{H(\tilde P)}{1-|\lambda|},
$$
(see for example \cite[Lemma~3.5]{S}) where $H(f)$ denotes the naive height of a polynomial $f$ i.e. the maximum of the modulus of its coefficients. Hence
$$
\tilde r_{\tilde Q} \geq \frac{1-|\lambda|}{1-|\lambda| + H(\tilde P)}.
$$
Recall (see \cite[Lemma~1.2~(c)]{KPS}) that for two polynomials $f,g$ we have
$$
H(f \circ g) \leq H(f)H(g)^{\deg f} 2^{(\deg f)(\deg g +1)}.
$$
In particular,
$$
H(\unipol^t) \leq H(\unipol^{t-1})H(\unipol)^{2^{t-1}}2^{2^t+2^{t-1}} = H(\unipol^{t-1})\max \{1, |c| \}^{2^{t-1}}2^{2^t+2^{t-1}},
$$
and so inductively we obtain
\begin{align*}
H(\unipol^t) & \leq \max \{1, |c| \}^{2^{t-1}+\cdots+2+1} 2^{2^{t+1}+2^{t-1}+\cdots+2^3+2} \\
& = \max \{1, |c| \}^{2^t-1} 2^{2^{t+1} + 2^t - 6}.
\end{align*}
Hence,
\begin{align*}
H( \tilde P) & \leq H(\unipol^t(z-\alpha)) + |\alpha| \leq \max\{1,|\alpha|\} H(\unipol^t) 2^{2^t+1} + |\alpha| \\
& \leq 2 \left( \max \{1, |c| \}^{2^t-1} 2^{2^{t+1} + 2^t - 6} + 1 \right) \\
& \leq 2 \left(2^{2^{t+2}-7}+1 \right) \leq 2^{2^{t+2}-5}
\end{align*}
whence
\begin{equation} \label{eq:ConvRad2}
d(\alpha, \partial U) \geq \tilde r_{\tilde Q} \geq \frac{1-|\lambda|}{2^{2^{t+2}-5}}.
\end{equation}

In conclusion, plugging \eqref{eq:ConvRad1} and \eqref{eq:ConvRad2} into \eqref{eq:boundo}, we obtain the following result.

\begin{prop} \label{prop:InMandelBound}
Let $\unipol(z)=z^2+c$, with $c$ lying in a hyperbolic component of the Mandelbrot set of period $t$, and suppose the period $t$ attracting cycle of $\unipol$ has multiplier $\lambda \neq 0$. Then there is a disk about 0 of radius at least
$$
\frac{|\lambda| \left( |\lambda|^2 - 1 + \sqrt{1-|\lambda|^2} \right) \left( |\lambda| - 1 + \sqrt{1- |\lambda|^2} \right) (1- |\lambda|)}{ C_3 \left( |\lambda| + 1 - \sqrt{1-|\lambda|^2} \right)^3}
$$
containing no preperiodic points for $\unipol$, where
$$
C_3 = \begin{cases} 
1 & \text{if } t=1, \\
2^{2^{t+2}-5} & \text{if } t > 1. \end{cases}
$$ 
\end{prop}

Note that the bound given in Proposition~\ref{prop:InMandelBound} is asymptotically equivalent to $|\lambda|/(2C_3)$ as $|\lambda| \to 0$, and to $(1-|\lambda|)^2/(4C_3)$ as $|\lambda| \to 1$.

In \cite[Theorem~2]{I2}, given a rational function $\rat$ with a periodic cycle, Ingram relates the \emph{critical height} of $\rat$ (in the case $\rat=\unipol$, this is proportional to the canonical height $\canheight{\unipol}(0)$ of the unique critical point) to the multiplier of the periodic cycle. We present below a version of this result, which will allow us to remove the dependence on $|\lambda|$ from Proposition~\ref{prop:InMandelBound}.

\begin{lemma} \label{lem:CritHeightVSMultiplierHeight}
Suppose $\unipol(z)=z^2+c$ has a $t$-cycle of multiplier $\lambda \neq 1$. Then
$$
h(\lambda) \leq C_4 \canheight{\unipol}(0) + C_5,
$$
where
\begin{equation*}
C_4 = t2^{2t} + C_1(2^{t+2}-1)(2^t+2)(2^t-1),
\end{equation*}
and
\begin{align*} 
C_5 & = (2^{t+2}-1) ( (2^t + 2) (2^t-1) \left( C_2 + 4 \log 8 \right) \notag\\
& \qquad \quad + (2^t+1)^2 \log 2 + \log (2^t+1)(2^t+2) ) + 2 \log 2^{t+1}(2^{t+1}-1)! \notag \\
& \qquad \qquad + \log 2 + 2^{t+1} \log \mathrm{lcm}(1,\ldots,2^t) + 2 \log \max \left \{8, 3^{2^t-1} \right \},
\end{align*}
where $C_1,C_2$ are the constants appearing in Corollary~\ref{cor:CanHeightBound1}.
\end{lemma}

\begin{proof}
Given a rational function $\rat : \P^1 \to \P^1$ of degree at most $d$, written in the form
$$
\rat(z) = \frac{c_0 + c_1 z + \cdots + c_d z^d}{c_{d+1} + c_{d+2}z + \cdots + c_{2d+1}z^d},
$$
we set
$$
h_{\mathrm{Hom}_d}(\rat) := \sum_{v \in M_K} \log \max \{ |c_0|_v , \ldots, |c_{2d+1}|_v \},
$$
which is well-defined by the product formula. By \cite[Proposition~5]{HS}, we have
$$
h_{\mathrm{Hom}_{d^n}}(\rat^n) \leq \left( \frac{d^n-1}{d-1} \right) h_{\mathrm{Hom}_d}(\rat) + d^2 \left( \frac{d^{n-1}-1}{d-1} \right) \log 8.
$$
In particular, for the unicritical polynomial $\unipol(z)=z^2+c$ we have $h_{\mathrm{Hom}_d}(\unipol) = h(c)$, and so
\begin{equation} \label{eq:HB1}
h_{\mathrm{Hom}_{d^t}}(\unipol^t) \leq (2^t-1) h(c) + 4 ( 2^{t-1}-1) \log 8.
\end{equation}

Also, define the \emph{critical height}
$$
\hat h_{\mathrm{crit}}(\rat) := \sum_{\alpha \in \P^1} (e_\alpha(\rat)-1) \canheight{\rat}(\alpha),
$$
where $e_\alpha(\rat) = \mathrm{ord}_\alpha(\rat(z)-\rat(\alpha))$ is the \emph{ramification index} of $\rat$ at $\alpha$. By \cite[Proposition~6.25~(b)]{Si2}, we have $\hat h_{\mathrm{crit}}(\rat^n) = n \hat h_{\mathrm{crit}}(\rat)$ for any $n \geq 1$. Note also that 
$$
\hat h_{\mathrm{crit}}(\unipol) = \canheight{\unipol}(0),
$$
and so
\begin{equation} \label{eq:HB2}
\hat h_{\mathrm{crit}}(\unipol^t) = t \canheight{\unipol}(0).
\end{equation}
Moreover, it is easy to show that $\hat h_{\mathrm{crit}}$ is invariant under change of coordinates.

By \cite[Lemma~11]{I2}, there exists a rational function $\psi$, conjugate to $\unipol^t$ by a linear fractional transformation, such that $\psi(0)=0$ with multiplier $\lambda$, $\psi(\infty)=\infty$ and
\begin{align} \label{eq:HB3}
h_{\mathrm{Hom}_{2^t}}(\psi) & \leq (2^t+2) h_{\mathrm{Hom}_{2^t}}(\unipol^t) + (2^t+1)^2 \log 2 + \log(2^t+1)(2^t+2) \notag \\
& \leq (2^t+2) \left( (2^t-1) h(c) + 4 ( 2^{t-1}-1) \log 8 \right) \notag \\
& \qquad \qquad + (2^t+1)^2 \log 2 + \log(2^t+1)(2^t+2),
\end{align}
where the last inequality follows from \eqref{eq:HB1}.

Now, following the proof of \cite[Lemma~12]{I2} with $k=1$ and using \eqref{eq:HB2}, we have
\begin{align} \label{eq:HB4}
t2^{2t}(t & -1) \canheight{\unipol}(0) = 2^{2t} \hat h_{\mathrm{crit}}(\unipol^n) = 2^{2t} \hat h_{\mathrm{crit}}(\psi) \notag \\
 & \geq h(\lambda) - (2^{t+2}-1) h_{\mathrm{Hom}_{2^t}}(\psi) - 2 \log \left( 2^{t+1}(2^{t+1}-1)! \right) \notag \\
& \qquad \qquad - \log 2 - 4 \log \mathrm{lcm}(1,\ldots,2^t) - 2 \log \max \{ 8, 3^{2^t-1} \} \notag \\
& \geq h(\lambda) - (2^{t+2}-1) ( (2^t+2) \left( ( 2^t-1) h(c) + 4 ( 2^{t-1}-1) \log 8 \right) \notag \\
& \: + (2^t+1)^2 \log 2 + \log(2^t+1)(2^t+2) ) - 2 \log \left( 2^{t+1}(2^{t+1}-1)! \right) - \log 2 \notag \\
& \qquad \qquad - 4 \log \mathrm{lcm}(1,\ldots,2^t) - 2 \log \max \{ 8, 3^{2^t-1} \},
\end{align}
with the last inequality coming from \eqref{eq:HB3}. Recall from Corollary~\ref{cor:CanHeightBound1} that $h(c) \leq C_1 \canheight{\unipol}(0) + C_2$. Plugging this into \eqref{eq:HB4} and rearranging completes the proof.
\end{proof}

\begin{prop} \label{prop:DeltavBoundParamIhHypComp}
Suppose $v \in M_K^\infty$ is such that $c$ lies in a period $t$ hyperbolic component of the $v$-adic Mandelbrot set. Then
$$
- \log \delta_v \leq A_{\infty,2} + B_{\infty,2} \canheight{\unipol}(0),
$$
where
$$
A_{\infty,2} := C_5 + \log(20) + \log(C_3) \quad \text{and} \quad B_{\infty,2} := C_4.
$$
\end{prop}

\begin{proof}
First suppose that $0.1 \leq |\lambda| \leq 0.7$. Define
$$
g(x) := \frac{x(x^2-1+\sqrt{1-x^2})(x-1+\sqrt{1-x^2})(1-x)}{(x+1-\sqrt{1-x^2})^3}, \quad 0 < x \leq 1. 
$$
Using elementary calculus, it is easy to see that on $[0.1,0.7]$, $g(x)$ attains a minimum at $x=0.7$, and so by Proposition~\ref{prop:InMandelBound}
$$
\delta_v \geq \frac{g(0.7)}{C_3} >  \frac{1}{60C_3},
$$
whence
$$
-\log \delta_v \leq \log(60)+\log(C_3) < A_{\infty,2} + B_{\infty,2} \canheight{\unipol}(0).
$$
Note that $g(x)/x$ is decreasing on $(0,1]$, and $g(0.1)/0.1 > 1/3$, so if $|\lambda| < 0.1$, then
$$
\delta_v =  \frac{g(|\lambda|)}{C_3 |\lambda|} |\lambda| \geq \frac{g(0.1)}{0.1 C_3} |\lambda| \geq \frac{|\lambda|}{3C_3}.
$$
It follows from \eqref{eq:FundamentalHeightInequality}  and Lemma~\ref{lem:CritHeightVSMultiplierHeight} that
\begin{align*}
- \log \delta_v & \leq \log(3C_3)-\log |\lambda| \leq \log(3C_3)+h(\lambda) \\
& \leq C_4 \canheight{\unipol}(0) + C_5 + \log 3 + \log C_3 < A_{\infty,2} + B_{\infty,2} \canheight{\unipol}(0).
\end{align*}
Moreover, $g(x)/(1-x)^2$ is increasing on $(0,1]$, and $g(0.7)/0.3^2 > 1/5$, so if $|\lambda| > 0.7$, then
$$
\delta_v = \frac{g(|\lambda|)}{C_3 (1-|\lambda|)^2} (1-|\lambda|)^2 \geq \frac{g(0.7)}{ 0.3^2 C_3} (1-|\lambda|)^2 \geq \frac{|\lambda-\xi|^2}{10C_3},
$$
where $\xi$ is a root of unity with $|\lambda - \xi| \leq \sqrt{2}(1-|\lambda|)$, which exists since the roots of unity are dense in the unit circle. Thus, using \eqref{eq:FundamentalHeightInequality}, \eqref{eq:HeightSumInequality1} and Lemma~\ref{lem:CritHeightVSMultiplierHeight},
\begin{align*}
- \log \delta_v & \leq \log (10C_3) - \log |\lambda-\xi| \leq \log(10C_3) + h(\lambda-\xi) \\
& \leq \log(10C_3) + \log 2 + h(\xi) + h(\lambda) \\
& \leq C_4 \canheight{\unipol}(0) + C_5 + \log(20) + \log(C_3) \leq A_{\infty,2} + B_{\infty,2} \canheight{\unipol}(0),
\end{align*}
as desired.
\end{proof}

\section{Bounds on preperiodic points at places of good reduction} \label{sec:nonarch}

Let $v$ be a finite place of $K$ lying over a rational prime $p$, write $| \cdot | = | \cdot |_v$, $\cO = \cO_v$, $\m=\m_v$ and $k=k_v$. Note that $k$ is a finite field of characteristic $p$, say with $q$ elements. Assume $0$ is not preperiodic for $\unipol$, and that $|c| \leq 1$, so that $\unipol$ has good reduction at $v$. In this section, for $z \in \C_v$ we use the notation $O(z)$ to denote an element $x \in \C_v$ with $|x| \leq |z|$. 

Since there are only finitely many residue classes modulo $v$, the unit disk $D(0,1) \subset \C_v$ must be preperiodic under $\unipol$. In particular, we can take minimal $n \geq 1$, $m \geq 0$ with $n + m \leq q$ such that $\unipol^{n+m}(D(0,1)) = \unipol^m(D(0,1)) = D(x,1)$ where $x := \unipol^m(0)$ (the last equality follows for example from \cite[Theorem~3.15]{B}. Set $y := \unipol^n(x)-x$ so that $|y| < 1$, and on $D(x,1)$, write
\begin{equation} \label{eq:NewtonPoly}
\unipol^n(z)-x = c_0 + c_1(z-x)+\cdots + c_{2^n}(z-x)^{2^n}.
\end{equation}
Then $c_0 = y$, so $|c_0| < 1$, and
\begin{equation} \label{eq:NewtonPolyC1}
c_1 = (\unipol^n)'(x) = 2^n x \unipol(x) \cdots \unipol^{n-1}(x).
\end{equation}

\subsection{Attracting case}
We have $|c_1| < 1$ if either $m = 0$ or $|2| < 1$. In this case, since $c_{2^n}=1$, $\unipol^n-x$ has Weierstrass degree at least 2 on $D(x,1)$, and so by \cite[Theorem~4.18]{B}, $D(x,1)$ is an attracting component for $\unipol$ with a unique attracting periodic point $b$ of exact period $n$. We first produce an iterate of $\unipol^n$ which sends $x$ sufficiently close to the attracting periodic point $b$.

\begin{lemma} \label{lem:AttractSmallCoeffs}
There exists an integer $\ell$ with
$$
\ell \leq \left \lfloor \frac{\log \left( \frac{n \log |2|}{\log |y|}+ 1 \right)}{\log 2} \right \rfloor + 1
$$
such that $\delta := |\unipol^{ \ell n}(x)-b| < |(\unipol^n)'(b)|$. In particular, we can take $\ell=1$ if $|2|=1$. Moreover, $\delta$ satisfies
$$
-\log \delta \leq 2^{r_1}h(c)+(2^{r_1-1}+1)\log 2,
$$
where
$$
r_1 := \begin{cases} q-1 & |2|=1, \\ q \left( \left \lfloor \frac{\log \left( \frac{q \log |2|}{\log |\pi|}+1 \right)}{\log 2} \right \rfloor + 1 \right)-1 & |2| < 1. \end{cases}
$$
\end{lemma}

\begin{remark} \label{rem:AttractSmallCoeffs}
Note that if we write
$$
\unipol^{\ell n}(z) - \unipol^{\ell n}(x) = b_0 + b_1 (z - \unipol^{\ell n}(x)) + \cdots + c_{2^{\ell n}} (z - \unipol^{\ell n}(x))^{2^{\ell n}},
$$
then $b_0 = \unipol^{2 \ell n}(x) - \unipol^{\ell n}(x)$ and
$$
b_1 = (\unipol^{\ell n})'(\unipol^{\ell n}(x)) = 2^{\ell n} \unipol^{\ell n}(x) \unipol^{\ell n +1}(x) \cdots \unipol^{2 \ell n -1}(x)
$$
satisfies $|b_1| < 1$, and so
$$
b - \unipol^{\ell n}(x) = \unipol^{\ell n}(b) - \unipol^{\ell n}(x) = b_0 + b_1(b-\unipol^{\ell n}(x)) + O \left( (b - \unipol^{\ell n}(x))^2 \right),
$$
implies that we can calculate
$$
|\unipol^{\ell n}(x) - b| = |\unipol^{2 \ell n}(x) - \unipol^{\ell n}(x)|.
$$ 
\end{remark}

\begin{proof}
Looking to the Newton polygon of \eqref{eq:NewtonPoly}, we see that $|b-x|=|y|$. Write $x = b+w$ with $|w| = |y|$. Then
$$
\unipol^n(x) = \unipol^n(b) + (\unipol^n)'(b)w + O(w^2) = b + (\unipol^n)'(b)w +O(w^2),
$$
and so
$$
|\unipol^n(x)-b| \leq \max \{ |(\unipol^n)'(b)||y|, |y|^2 \}.
$$
Continuing by induction, we obtain
\begin{equation} \label{eq:7.1Bound}
|\unipol^{\ell n}(x)-b| \leq \max \{ |(\unipol^n)'(b)||y|, |y|^{2^\ell} \}.
\end{equation}
Now, if $m=0$, then $|b|=|y|<1$. For $1 \leq i < n$, we have
$$
\unipol^i(b) = \unipol^i(0) + O(b^d),
$$
which gives $|\unipol^i(b)| = 1$, since by definition $|\unipol^i(0)|=1$ for $i < n$. On the other hand, if $m > 0$, then $|b|=1$ and so $|\unipol^i(b)|=1$ for $0 \leq i < n$, as otherwise, say $|\unipol^i(b)| < 1$, we would have
$$
\unipol^i(b) = \unipol^{n+i}(b) = \unipol^n(0) + O(\unipol^i(b)^2),
$$
giving $|\unipol^i(b)|=1$, a contradiction. In either case, we have
$$
|(\unipol^n)'(b)| = |2|^n \left| b \unipol(b) \cdots \unipol^{n-1}(b) \right| \geq |2|^n |y|.
$$
Note in particular that if $m > 0$, then $|(\unipol^n)'(b)|=|2|^n$, so by \eqref{eq:7.1Bound} we can take $\ell = 1$ if $|2|=1$. Moreover, if $m=0$, then
$$
b = \unipol^n(b) = \unipol^n(0) + O(b^2) = \unipol^n(x)+O(b^2),
$$
so $|\unipol^n(x)-b| \leq |y|^2$, and we can again take $\ell = 1$ if $|2|=1$.

In general, by \eqref{eq:7.1Bound} for
$$
\ell > \frac{\log \left( \frac{n \log |2|}{\log |y|}+1 \right)}{\log 2}
$$
we have
$$
|\unipol^{\ell n}(x)-b| \leq |2|^n |y| \leq |(\unipol^n)'(b)|.
$$
Note that since $b$ is periodic for $\unipol$, $\canheight{\unipol}(b) = 0$ and so from \eqref{eq:CanHeightConst}, 
$$
h(b) \leq h(c)+\log 2.
$$
Now, by \eqref{eq:FundamentalHeightInequality} and \eqref{eq:HeightSumInequality1},
\begin{align*}
- \log \delta & \leq - \log | \unipol^{\ell n+m}(0) - b | \leq h(\unipol^{\ell n+m-1}(c)-b) \\
& \leq h(\unipol^{\ell n+m-1}(c)) + h(b) + \log 2 \\
& \leq (2^{\ell n+m-1}-1) h(c) + (2^{\ell n+m-2}-1) \log 2 + h(c)+2 \log 2 \\
& = 2^{\ell n+m-1} h(c) + (2^{\ell n+m-2}+1)\log 2.
\end{align*}
The result follows, noting that
$$
\ell n+m = (\ell -1)n+(n+m) \leq \ell q \leq q \left( \left \lfloor \frac{\log \left( \frac{q \log |2|}{\log |\pi|}+1 \right)}{\log 2} \right \rfloor + 1 \right)
$$
since $n+m \leq q$ and $|y| \geq |\pi|$.
\end{proof}

Now, for points $z$ close to 0, we can precisely control the distance from iterates of $z$ under $\unipol$ to $b$, ruling out the possibility that $z$ is preperiodic. 

\begin{prop} \label{prop:AttractNoPrePer}
Let $\delta$ be as in Lemma~\ref{lem:AttractSmallCoeffs}. Then $D(0,\delta)$ contains no preperiodic points for $\unipol$.
\end{prop}

\begin{proof}
Let $z \in D(0,\delta)$. Then
$$
\unipol^{\ell n+m}(z) - b = \unipol^{\ell n+m}(0) - b + O(z^d),
$$
and so $|\unipol^{\ell n+m}(z)-b| = |\unipol^{\ell n}(x)-b|=\delta$. Write $\unipol^{\ell n+m}(z) = b+w$ with $|w| = \delta$. Then
$$
\unipol^{(\ell +1)n+m}(z) = \unipol^n(b) + (\unipol^n)'(b)w + O(w^2).
$$
Since $|w| = \delta < |(\unipol^n)'(b)|$, we have
$$
|\unipol^{(\ell +1)n+m}(z)-b| = |(\unipol^n)'(b)| \delta.
$$
We can therefore proceed by induction to obtain
$$
|\unipol^{(\ell+j)n+m}(z)-b| = |(\unipol^n)'(b)|^j \delta \neq 0
$$
for any $j \geq 1$. But $\unipol^{(\ell+j)n+m}(z) \to b$ as $j \to \infty$ since $b$ is attracting, so we conclude that $z$ is not preperiodic for $\unipol$.
\end{proof}

\subsection{Indifferent case}

The only case remaining is where $|2|=1$, and $D(0,1) \subset \C_v$ is strictly preperiodic, say with (minimal) $m,n \geq 1$ such that $\unipol^{m+n}(D(0,1)) = \unipol^m(D(0,1)) = D(x,1)$, where $x = \unipol^m(0)$. In this case, since $|\unipol^i(x)|=1$ for all $i$, referring again to \eqref{eq:NewtonPoly} and \eqref{eq:NewtonPolyC1}, we have $|c_1|=1$, and so $\unipol^n-x$ has Weierstrass degree 1 on $D(x,1)$. Hence by \cite[Theorem~4.18]{B}, $D(x,1)$ is an indifferent periodic component for $\unipol$, which is mapped bijectively onto itself by $\unipol^n$. 

For this case, we will use the following form of the Taylor expansion: Let $n \geq 1$, $z \in \C_v$, and set $w := \unipol^n(z)-z$. Then
\begin{align*}
\unipol^{2n}(z) = \unipol^n(z+w) & = \unipol^n(z) + (\unipol^n)'(z)w + O(w^2) \\
& = z + (1+ (\unipol^n)'(z))w + O(w^2).
\end{align*}
Continuing this inductively, we obtain for any integer $t \geq 1$,
$$
\unipol^{tn}(z) = z + (1+(\unipol^n)'(z) + \cdots + (\unipol^n)'(z)^{t-1})w + O(w^2).
$$
Note in particular that $\unipol^{tn}(z)-z=O(w)$, and so
\begin{align*}
\unipol^{tn+1}(z) = \unipol(z + \unipol^{tn}(z)-z) & = \unipol(z) + \unipol'(z) (\unipol^{tn}(z)-z) + O((\unipol^{tn}(z)-z)^2) \\
& = \unipol(z) + \unipol'(z) (1+(\unipol^n)'(z) + \cdots + (\unipol^n)'(z)^{t-1})w + O(w^2).
\end{align*}
Thus, by another induction, this time on $0 \leq s < n$, we obtain
\begin{equation} \label{eq:NonArchTaylor}
\unipol^{tn+s}(z) = \unipol^s(z) + \unipol'(z)^s (1+(\unipol^n)'(z) + \cdots + (\unipol^n)'(z)^{t-1})w + O(w^2),
\end{equation}
for all $t \geq 1$ and $0 \leq s < n$.

Note also the following.

\begin{lemma} \label{lem:GPVal}
Let $z \in \C_v$. If $|z-1| < |p|^{1/(p-1)}$, then
$$
\left| \frac{z^n-1}{z-1} \right| = |n|
$$
for $n = p$ and all integers $n$ with $p \nmid n$. If $|z-1| \geq |p|^{1/(p-1)}$, then
$$
\left| \frac{z^p-1}{z-1} \right| = |z-1|^{p-1}.
$$
\end{lemma}

\begin{proof}
Write $z=1+u$. If $|u| < |p|^{1/(p-1)}$, then
$$
|z^n-1| = |(1+u)^n - 1| = \left| nu +\binom{n}{2} u^2 + \cdots + n u^{n-1} + u^n \right|.
$$
If $p \nmid n$, then for $k > 1$
$$
\left| \binom{n}{k} u^k \right| \leq |u|^k < |u| = |nu|,
$$
and so $|z^n-1| = |nu| = |u|$ gives $|(z^n-1)/(z-1)|=1$. If $n=p$, then
$$
\left| \binom{n}{k} u^k \right| = \begin{cases} |p| |u|^k, & 1 \leq k < n, \\ |u|^n, & k = n, \end{cases}
$$
so $|z^p-1| = |p||u|$, since $|u| < |p|^{1/(p-1)}$. We conclude that
$$
\left| \frac{z^p-1}{z-1} \right| = |p|.
$$
On the other hand, if $|u| \geq |p|^{1/(p-1)}$, it is easy to see that
$$
\max_{1 \leq k \leq p} \left|\binom{p}{k} \right| |u^k| = |u|^p,
$$
and so
$$
\left| \frac{z^p-1}{z-1} \right| = |z-1|^{p-1}.
$$
\end{proof}

The indifferent case is more complicated, but we can proceed along similar lines, in the spirit of how Rivera-Letelier \cite[\textsection 3.2]{RL} proved the discreteness of Fatou preperiodic points (see also \cite[\textsection 10.2]{B}). We say that a point $a \in \C_v$ is \emph{quasi-periodic} under $\unipol$ if there is a radius $r > 0$ and an integer $m \geq 1$ such that the \emph{iterative logarithm} of $\unipol$,
$$
L_{\unipol}(z) := \lim_{\ell \to \infty} \frac{\unipol^{mp^\ell}(z)-z}{mp^\ell}
$$
converges uniformly on $\overline D(a,r)$. Rivera-Letelier proved that for a quasi-periodic point $a$, $L_{\unipol}(a) = 0$ if and only if $a$ is an indifferent periodic point, and moreover that $L_{\unipol}$ is a power series converging on a neighbourhood of $a$. Since the zeros of a power series are isolated, this can be used to prove the discreteness of preperiodic points in the Fatou set; see \cite[Proposition~2]{P}. In our case, all points in $D(x,1)$ are quasi-periodic, and so the appropriate analogues of Lemma~\ref{lem:AttractSmallCoeffs} and Proposition~\ref{prop:AttractNoPrePer} involve controlling the rate at which $\unipol^{mp^\ell}(z)$ converges to $z$ for some appropriately chosen $m$ and points $z$ sufficiently close to $x$.

\begin{lemma} \label{lem:SmallCoeffs}
There exist non-negative integers $j,\ell$ with $1 \leq j \leq q-1$, and 
$$
0 \leq \ell \leq \left \lfloor \frac{\log \left( \sqrt{5} \frac{\log |p|}{\log |\pi|} + \frac{1}{2} \right)}{\log \phi} \right \rfloor,
$$
where $\phi$ denotes the golden ratio, such that $|\unipol^{jnp^\ell}(x)-x| =: \delta^2 < |p|$ and $|(\unipol^{jnp^\ell})'(x)-1| < |p|^{1/(p-1)}$. Also, $\delta$ satisfies
$$
-\log \delta \leq (2^{r_2}+2^{q-2}-1)h(c)+(2^{r_2-1}+2^{q-3}-1) \log 2,
$$
where
$$
r_2 := q(q-1) \left(\sqrt{5} \frac{\log |p|}{\log |\pi|}+\frac{1}{2} \right)^{\frac{\log p}{\log \phi}} - 2.
$$
Moreover, in the case $K = \Q$ we can take $\ell = 1$, and we have
$$
- \log \delta \leq \left( 2^{p^2(p-1)-2}+2^{p-2}-1 \right) h(c) + \left( 2^{p^2(p-1)-3}+d^{p-3}-1 \right) \log 2.
$$
\end{lemma}

\begin{proof}
Recall that we write $\unipol^n(x)=x+y$ with $|y| \leq |\pi| < 1$. We have
\begin{align*}
\unipol^{2n}(x) = \unipol^n(x+y) & = \unipol^n(x) + (\unipol^n)'(x)y + O(y^2) \\
& = x + y(1+(\unipol^n)'(x)) + O(y^2).
\end{align*}
If $|(\unipol^n)'(x)-1| < 1$, let $j = 1$. Otherwise, let $j = q-1$. Since $|(\unipol^n)'(x)|=1$, and \eqref{eq:NonArchTaylor} gives
$$
\unipol^{s+tn}(x) \equiv \unipol^s(x) \pmod{\m}
$$ 
for $0 \leq s < n$ and $1 \leq t < j$, we have
\begin{align*}
(\unipol^{jn})'(x) & = 2^{jn} x \unipol(x) \cdots \unipol^{jn-1}(x)  \\
& \equiv 2^{(q-1)n} (x \unipol(x) \cdots \unipol^{n-1}(x))^{q-1} \\
& \equiv (2^n)^{q-1} (\unipol^n)'(x)^{q-1} \equiv 1 \pmod{\m},
\end{align*}
recalling that the residue field $k = \cO/\m$ has $q$ elements.

We show by induction that for all $i \geq 0$,
$$
|\unipol^{p^i jn}(x)-x| \leq \max \left \{ |p|, |\pi|^{F_{i+1}} \right \}, \quad |(\unipol^{p^i jn})'(x)-1| \leq \max \left \{ |p|, |\pi|^{F_i} \right \},
$$
where $F_i$ denotes the $i$-th Fibonacci number. Note that in the case $K=\Q$, we have $|\pi|=|p|$, and so this implies we can take $\ell=1$. The basis case follows from above. Let $i \geq 0$ and write $\unipol^{p^i jn}(x)=x+\beta$ with $|\beta| \leq \max \{ |p|, |\pi|^{F_{i+1}} \}$ and $(\unipol^{p^i jn})'(x) = 1 + \alpha$ with $|\alpha| \leq \max \{ |p|, |\pi|^{F_i} \}$. Then using \eqref{eq:NonArchTaylor} with $s=0$ and $t=p$, but with $p^i jn$ in place of $n$, we have from Lemma~\ref{lem:GPVal}
\begin{align*}
\left| \unipol^{p^{i+1} jn}(x) - x \right| & = \left| \beta \left( \frac{(\unipol^{p^i jn})'(x)^p-1}{(\unipol^{p^i jn})'(x)-1} \right) + O(\beta^2) \right| \\
& \leq \max \{ |p|, |\beta||\alpha|^{p-1}, |\beta|^2 \} \\
& \leq \max \{ |p|, |\pi|^{F_{i+1} + (p-1)F_i}, |\pi|^{2F_{i+1}} \} \leq \max \{ |p|, |\pi|^{F_{i+2}} \}.
\end{align*}
Additionally, since \eqref{eq:NonArchTaylor} gives $\unipol^{s+tp^ijn}(x) = \unipol^s(x) + O(\beta)$ for $0 \leq s < p^i jn$, $1 \leq t < p$, we have
\begin{align*}
(\unipol^{p^{i+1}jn})'(x) & = 2^{p^{i+1}jn} x \unipol(x) \cdots \unipol^{p^{i+1}jn-1}(x)  \\
& = 2^{p^{i+1}jn} \left(x \unipol(x) \cdots \unipol^{p^i jn-1}(x) \right)^{p} + O(\beta) \\
& = (\unipol^{p^i jn})'(x)^p + O(\beta),
\end{align*}
and so by Lemma~\ref{lem:GPVal},
\begin{align*}
|(\unipol^{p^{i+1} jn})'(x)-1| & = |(\unipol^{p^i jn})'(x)^p - 1 + O(\beta)| \\
& \leq \max \{ |p||\alpha|, |\beta|, |\alpha|^p \} \\
& \leq \max \{ |p|, |\pi|^{F_{i+1}}, |\pi|^{p F_i} \} \leq \max \{ |p|, |\pi|^{F_{i+1}} \}.
\end{align*}
The bound on $\ell$ follows by noting that
$$
F_\ell = \left \lfloor \frac{\phi^\ell}{\sqrt{5}} - \frac{1}{2} \right \rfloor.
$$

Now, from \eqref{eq:FundamentalHeightInequality} and \eqref{eq:HeightSumInequality1} we have
\begin{align*}
- \log \delta^d & = - \log | \unipol^{p^\ell jn}(x) - x | \leq h(\unipol^{p^\ell jn}(x)-x)| \\
& \leq h(\unipol^{p^\ell j n + m - 1}(c)) + h(\unipol^{m-1}(c)) + \log 2 \\
& \leq  \left( 2^{p^\ell jn + m - 1}+2^{m-1}-2 \right) h(c) + \left(2^{p^\ell jn+m-2}+2^{m-2}-2 \right) \log 2.
\end{align*}
The result follows, noting that $j \leq q - 1$, $n + m \leq q$ and
$$
p^{\ell} \leq \left(\sqrt{5} \frac{\log |p|}{\log |\pi|}+\frac{1}{2} \right)^{\frac{\log p}{\log \phi}}.
$$
The bound for the case $K=\Q$ comes from taking $\ell = 1$ and noting that $q=p$.
\end{proof}

\begin{prop} \label{prop:IndiffNoPer}
Let $\delta$ be as in Lemma~\ref{lem:SmallCoeffs}. Then there are no preperiodic points for $\unipol$ in the disk $D(x,\delta^d)$.
\end{prop}

\begin{proof}
Let $j,\ell$ be as in Lemma~\ref{lem:SmallCoeffs}, write $N = jp^\ell$ and
$$
\unipol^N(x) = x + a
$$
with $|a| = \delta^d$. Let $z \in D(x,\delta^d)$ and write $z=x+w$ with $|w| < \delta^d$. Then
\begin{align*}
\unipol^N(z) = \unipol^N(x+w) & = \unipol^N(x)+ (\unipol^N)'(x)w + O(w^2) \\
& = x+a+w+ \left( (\unipol^N)'(x)-1 \right)w + O(w^2) \\
& =: z + Y,
\end{align*}
with $|Y| = |a| = \delta^d$. Note that
\begin{align*}
(\unipol^N)'(z) & = 2^N z \unipol(z) \cdots \unipol^{N-1}(z) \\
& = 2^N (a+O(x))(\unipol(a)+O(x)) \cdots (\unipol^{N-1}(a)+O(x)) \\
& = (\unipol^N)'(a)+O(x),
\end{align*}
and so
$$
|(\unipol^N)'(z)-1| = |(\unipol^N)'(a)-1+O(x)| \leq |p|^{1/p-1}.
$$
Let $0 \leq s < N$. By \eqref{eq:NonArchTaylor} and Lemma~\ref{lem:GPVal}, for any integer $t > 0$ with $p \nmid t$ we have
\begin{align*}
|\unipol^{s+tN}(z)-\unipol^{s}(z)| & = \left|Y \unipol'(z)^s \left( \frac{(\unipol^N)'(z)^t-1}{(\unipol^N)'(z)-1} \right) + O(Y^2) \right| \\
& = |Y|.
\end{align*}
Furthermore, we have
\begin{align*}
|\unipol^{s+pN}(z)-\unipol^{s}(z)| & = \left|Y \unipol'(z)^s \left( \frac{(\unipol^N)'(z)^p-1}{(\unipol^N)'(z)-1} \right) + O(Y^2) \right| \\
& = |Y||p|,
\end{align*}
since $|Y| < |p|$. Now,
\begin{align*}
(\unipol^{pN})'(z) & = 2^{pN} z \unipol(z) \cdots \unipol^{pN-1}(z) \\
& = 2^{pN} \left( z \unipol(z) \cdots \unipol^{N-1}(z) \right)^{p}+O(Y) \\
& = (\unipol^N)'(z)^p + O(Y),
\end{align*}
so
$$
|(\unipol^{pN})'(z)-1| = |(\unipol^N)'(z)^p-1+O(Y)| \leq |p|,
$$
and hence we may proceed by induction to obtain
$$
|\unipol^{s+tN}(z)-\unipol^{s}(z)| = |Y||p|^{e_t},
$$
where $p^{e_t} \mathrel\Vert t$. In particular $\unipol^{tN}(z) \neq z$ for all $t$, and so $z$ is not periodic for $\unipol$. Since $\unipol^n$ maps $D(x,1)$ 1-1 onto itself, $D(x,1)$ contains no strictly preperiodic points, and so the proof is complete.
\end{proof}

\begin{corollary} \label{cor:IndiffNoPrePer}
Let $\delta$ be as in Lemma~\ref{lem:SmallCoeffs}. Then there are no preperiodic points for $\unipol$ in the disk $D(0,\delta)$.
\end{corollary}

\begin{proof}
By Proposition~\ref{prop:IndiffNoPer}, the preimage of $D(x,\delta^2)$ under $\unipol^m$ contains no preperiodic points. This preimage contains a disk $D(0,\varepsilon)$ about 0, which maps onto $D(x,\delta^2)$ under $\unipol^m$. Write 
$$
\unipol^m(z) = a_0 + a_2z^2 + \cdots + a_{2^m} z^{2^m},
$$ 
and note that $|a_i| \leq 1$ for $0 \leq i \leq 2^m$. We have that the Weierstrass degree $w$ of $\unipol^m(z)-a_0$ on $D(0,\varepsilon)$ satisfies $w \geq 2$, and so by \cite[Theorem~3.15]{B} we have
$$
\varepsilon = \left(\frac{\delta^2}{|c_w|}\right)^{1/w} \geq \delta,
$$
as desired. 
\end{proof}

The following is immediate from the above results, noting for the case $K=\Q$ that for a prime $p$ of good reduction, $q=p$ and $|\pi_p|_p = 1/p$.

\begin{corollary} \label{cor:nonArchDeltavBound}
Let $v$ be a place of good reduction for $\unipol$ lying over a prime $p$, and such that the residue field $k_v$ has $q$ elements. Then
$$
- \log \delta_v \leq A_v + B_v \canheight{\unipol}(0)
$$
with
$$
A_v := (C_2+\log 2) 2^{r_v}, \qquad B_v := C_1 2^{r_v},
$$
where
\begin{align*}
r_v & = \max \Bigg \{ q \left( \frac{\log \left( \frac{q \log |2|_v}{\log |\pi_v|_v}+1 \right)}{\log 2} + 1 \right) - 1 , \\
& \qquad \qquad \qquad q(q-1) \left( \sqrt{5} \frac{\log |p|_v}{\log |\pi_v|_v} + \frac{1}{2} \right)^{\frac{\log p}{\log \phi}} - 1 \Bigg \},
\end{align*}
and $C_1,C_2$ are the constants appearing in Corollary~\ref{cor:CanHeightBound1}. In the case where $K = \Q$ and $v=p$ is a prime of good reduction we can take
$$
r_p = \max \left \{ p \left \lfloor \frac{\log \left( p v_p(2)+1 \right)}{\log 2} + 1 \right \rfloor - 1 , p^2(p-1)-1 \right \},
$$
where $v_p(2)$ denotes the $p$-adic order of $2$.

\end{corollary}

\section{Proof of Theorem~\ref{thm:UniformMain} and related results} \label{sec:Unif}

Let us first summarise our results relating the quantities appearing in Theorem~\ref{thm:main} to the canonical height $\canheight{\unipol}(\alpha)$ in the case where $\alpha = 0$ is the critical point of a unicritical poiynomial $\unipol(z)=z^2+c$.

\begin{theorem} \label{thm:deltavUnifBounds}
Let $\unipol(z)=z^2+c$, $c \in K$ and $v \in M^{0,\unipol}_{K,\mathrm{good}}$. Then there exist explicit constants $A_v, B_v > 0$, depending only on $K$, $v$ such that
$$
- \log \delta_{\unipol,v}(0) \leq A_v + B_v \canheight{\unipol}(0)
$$
Moreover, let $\varepsilon > 0$, $t \geq 1$. Then there exist explicit constants $A_\infty, B_\infty > 0$, depending only on $K$, $\varepsilon$ and $t$ such that if $v$ is an archimedean place of $K$ such that either $\loccanheight{v}{\unipol}(0) \geq \varepsilon$ or $c$ lies is a hyperbolic component of the $v$-adic Mandelbrot set of period at most $t$, then
$$
- \log \delta_{\unipol,v}(0) \leq A_\infty + B_\infty \canheight{\unipol}(0).
$$
Finally, there exist constants $A_\kappa, B_\kappa$, and $C_0$ depending only on $K$ such that
$$
\frac{4}{\log 2} \canheight{\unipol}(0) \leq \frac{1}{\kappa} \leq A_\kappa + B_\kappa \canheight{\unipol}(0),
$$
and
$$
\canheight{\unipol}(0) \geq C_0
$$
for all $c \in K$.
\end{theorem}

\begin{proof}
For places of good reduction, see Corollary~\ref{cor:nonArchDeltavBound}. For archimedean places, see Corollary~\ref{cor:ArchOutParamDeltavBound} and Proposition~\ref{prop:DeltavBoundParamIhHypComp} and set $A_\infty := \max \{ A_{\infty, 1}, A_{\infty,2} \}$ and $B_\infty := B_{\infty,2}$. By \eqref{eq:HolderExponent}, for all $v \in M_K$ we have H\"{o}lder exponents
$$
\kappa_v \geq \frac{\log 2}{\log 6 + 4 \log^+ |c|_v},
$$
so by \eqref{eq:FundamentalHeightInequality},
\begin{equation} \label{eq:Kap}
\kappa := \frac{\log 2}{\log 6 + 4h(c)}
\end{equation}
satisfies $\kappa \leq \inf_{v \in V} \kappa_v$. Then
$$
\frac{1}{\kappa} = \frac{\log 6}{\log 2} + \frac{4 h(c)}{\log 2} \geq \frac{4}{\log 2} \canheight{\unipol}(0)
$$
by \eqref{eq:CanHeightConst} with $\beta = c$, and we can take
\begin{equation} \label{eq:aKapbKap}
A_\kappa := 1 + \frac{\log 6 + 4 C_2}{\log 2}, \qquad B_\kappa := \frac{4 C_1}{\log 2},
\end{equation}
with $C_1$ and $C_2$ as in Corollary~\ref{cor:CanHeightBound1}. To conclude, $C_0$ is found in Corollary~\ref{cor:CanHeightLowBound}.
\end{proof}

This leads to the following result, which immediately implies Theorem~\ref{thm:UniformMain}.

\begin{theorem} \label{thm:UniformMainExplicit}
Let $K$ be a number field, $S$ a finite set of places of $K$ containing the archimedean ones, and let $\unipol(z)=z^d+c$, $d \geq 2$, $c \in K$. Suppose that the critical point $0$ is not preperiodic under $\unipol$. Further, let $\varepsilon > 0$, $t \geq 1$ and suppose that for all archimedean places $v \in M_K^\infty$, either $\loccanheight{v}{\unipol}(0) \geq \varepsilon$ or $c$ lies in a hyperbolic component of the Mandelbrot set of period at most $t$. Then
\begin{equation*}
|\relsupreper{\unipol}{S}{0}| \leq \max \Bigg \{ 29 \log 2, e^{1+\sqrt{2u}+u}, \frac{7+4|V|}{12 \log 2} \left( \frac{2 | \widetilde S|}{C_0} \right)^{\frac{\log 2}{4C_0}} e^{\frac{\log 2}{4} \left( \frac{A_\infty}{C_0} + B_\infty \right)} \Bigg \},
\end{equation*}
where $\widetilde S = S \setminus M^{0,\unipol}_{K,\mathrm{bad}}$, $V = M_K^\infty \cup M^{0,\unipol}_{K,\mathrm{bad}}$,
$$
u \leq \log(32 \pi (|V|+1) |\widetilde S|) + \log \left( \frac{A_\kappa A}{C_0^2} + \frac{A_\kappa B + A B_\kappa}{C_0} + B_\kappa B \right) - 1,
$$
and
$$
A := \sum_{v \in S \setminus M_K^\infty} A_v + |M_K^\infty| A_\infty, \quad \qquad B := \sum_{v \in S \setminus M_K^\infty} B_v + |M_K^\infty| B_\infty,
$$
with constants as in Theorem~\ref{thm:deltavUnifBounds}.
\end{theorem}

\begin{proof}
Take $\kappa$ as in \eqref{eq:Kap}. By \eqref{eq:HolderConstant}, for all $v \in V$,
$$
C_v \leq 4 \log 6 + 6 \log^+ |c|_v,
$$
and $C_v = 0$ for $v \notin V$, so
$$
C := 4|V| \log 6 + 6 h(c) \geq \sum_{v \in V} C_v.
$$
We have
$$
C\kappa = \frac{4|V| \log 12 + 6(\log 2) h(c)}{\log 6 + 4h(c)} \leq 4|V| \log 2 + \frac{3}{2} \log 2,
$$
so
\begin{equation*}
\frac{6C\kappa}{|V|+1} \leq \frac{24|V| \log 2}{|V|+1} + \frac{9 \log 2}{|V|+1} < 29 \log 2,
\end{equation*}
and note also that
\begin{equation} \label{eq:CKap2}
\frac{|V|+1}{2C\kappa} \leq \frac{|V|+1}{2} \left( \frac{1}{|V| \log 2} + \frac{2}{3 \log 2} \right) \leq \frac{7+4|V|}{12 \log 2}.
\end{equation}
Now, referring to the proof of Theorem~\ref{thm:main}, note that if $v$ is a place of bad reduction for $\unipol$, by Remark~\ref{rem:BadReductHeight}, $|z|_v = |c|_v^{1/2}$ for all $z \in \cP = \relsupreper{\unipol}{S}{0}$, and so by Lemma~\ref{lem:BadReductHeight}
$$
\Gamma_v = \log \max \{ |\alpha|_v, |c|_v^{1/2} \} = \loccanheight{v}{\unipol}(\alpha).
$$
Thus, in \eqref{eq:GamForm} and subsequent equations, we may replace $S$ with $\widetilde S = S \setminus M^{0,\unipol}_{K,\mathrm{bad}}$. In particular, an upper bound for $|\cP|$ is given on the right-hand side of \eqref{eq:PBound}, with $\widetilde S$ in place of $S$. We have
\begin{align*}
u & = \log \left( \frac{32\pi(|V|+1) \left( \sum_{v \in \widetilde S} |\log \delta_v|^{1/2} \right)^2}{\kappa \canheight{\unipol}(0)^2} \right)-1 \\
& \leq \log(32\pi(|V|+1)|\widetilde S|) - 1 + \\
& \log \left( \frac{(A_\kappa + B_\kappa \canheight{\unipol}(0)) \left( \sum_{v \in \widetilde S \setminus M_K^\infty} (A_v + B_v \canheight{\unipol}(0)) + |M_K^\infty|(A_\infty + B_\infty \canheight{\unipol}(0)) \right)}{\canheight{\unipol}(0)^2} \right),
\end{align*}
where the inequality follows from Theorem~\ref{thm:deltavUnifBounds} and the generalized mean inequality. Setting
$$
A  := \sum_{v \in \widetilde S \setminus M_K^\infty} A_v + |M_K^\infty| A_\infty, \quad \qquad B := \sum_{v \in \widetilde S \setminus M_K^\infty} B_v + |M_K^\infty| B_\infty,
$$
we have
\begin{align*}
u & \leq \log(32 \pi (|V|+1) |\widetilde S|) + \log \left( \frac{A_\kappa A + (A_\kappa B + A B_\kappa) \canheight{\unipol}(0) + B_\kappa B \canheight{\unipol}(0)^2}{ \canheight{\unipol}(0)^2} \right) - 1 \\
& \leq \log(32 \pi (|V|+1) |\widetilde S|) + \log \left( \frac{A_\kappa A}{C_0^2} + \frac{A_\kappa B + A B_\kappa}{C_0} + B_\kappa B \right) - 1,
\end{align*}
recalling that $\canheight{\unipol}(0) \geq C_0$ for all $c \in K$. Moreover, from \eqref{eq:CKap2} we have
\begin{align*}
& \frac{|V|+1}{2C\kappa} \left[ \frac{2}{\canheight{\unipol}(0)} \left( \left( \sum_{v \in M_K^\infty} \frac{1}{\delta_v} \right) + |\widetilde S \setminus M_K^\infty| \right) \right]^\kappa \\
&  \leq \frac{7+4|V|}{12 \log 2} \left( \frac{2}{\canheight{\unipol}(0)} \left( |M_K^\infty| e^{A_\infty + B_\infty \canheight{\unipol}(0)} + |\widetilde S \setminus M_K^\infty| \right) \right)^{\frac{\log 2}{4 \canheight{\unipol}(0)}} \\
& \leq \frac{7+4|V|}{12 \log 2} \left( \frac{2 | \widetilde S|}{C_0} \right)^{\frac{\log 2}{4C_0}} e^{\frac{\log 2}{4} \left( \frac{A_\infty}{C_0} + B_\infty \right)},
\end{align*}
again using Theorem~\ref{thm:deltavUnifBounds} and the fact that $\canheight{\unipol}(0) \geq C_0$. This completes the proof.
\end{proof}

Moreover, we have an explicit form of Corollary~\ref{cor:IntMain}

\begin{corollary} \label{cor:IntMain2}
Let $\unipol(z) = z^2 + c$ with $c \in \Z$ such that $0$ is not preperiodic for $\unipol$ i.e. $c \neq 0,-1,-2$. Let $S$ be a finite set of places of $\Q$ containing the archimedean place $\infty$. Then there are at most $e^{1+\sqrt{2u}+u}$ preperiodic points for $\unipol$ which are $S$-integral relative to $0$, where
$$
u \leq \log(64\pi |S|) + \log \Bigg( 16 \Bigg( 5 + \frac{20+5 \log 6}{\log 2}   \Bigg)  \sum_{p \in S \setminus \{ \infty \}} 2^{r_p} \Bigg)-1,
$$
with 
$$
r_p := \max \left \{ p \left \lfloor \frac{\log \left( p v_p(2)+1 \right)}{\log 2} + 1 \right \rfloor - 1 , p^2(p-1)-1 \right \},
$$
where $v_p(d)$ denotes the $p$-adic order of $d$.
\end{corollary}

\begin{example}
In the case $d = 2$ and $S = \{ \infty, 2 \}$, the above gives an upper bound of $451287434$ for the number of preperiodic points for a map of the form $\unipol(z)=z^2+c$, $c \in \Z$, which are $\{ \infty, 2 \}$-integral with respect to 0.
\end{example}

\begin{proof}
In this case, in Corollary~\ref{cor:CanHeightBound1} we can take $C_1 = 4$ and $C_2 = 0$, and in Corollary~\ref{cor:CanHeightLowBound} we can take $C_0 = 1/4$. Thus, looking to \eqref{eq:aKapbKap}, we have 
\begin{align*}
A_\kappa & = 1 + \frac{\log 6}{\log 2}, \\
B_\kappa & = \frac{16}{\log 2}.
\end{align*}
Now, all integers $c$ such that $0$ is not preperiodic for $\unipol$ lie outside the Mandelbrot set, so by Lemma~\ref{lem:JuliaSetDistance}, we can take $A_\infty = \log 2$ and $B_\infty = 0$. Thus, from Corollary~\ref{cor:nonArchDeltavBound}, we have
\begin{align*}
A & = \sum_{p \in S \setminus \{ \infty \}} A_p + A_\infty = \sum_{p \in S \setminus \{ \infty \} } (2^{r_p} \log 2) + \log 2, \\
B & = \sum_{p \in S \setminus \{ \infty \}} B_p + B_\infty = \sum_{p \in S \setminus \{ \infty \}} 2^{r_p+2} \geq A.
\end{align*}
The result follows from plugging these values into Theorem~\ref{thm:UniformMain}, and noting that $e^{1+\sqrt{2u}+u}$ attains the maximum of the bound therein.
\end{proof}

\section{$S$-units in dynamical sequences} \label{sec:Sunit}

To prove Theorem~\ref{thm:SunitBound}, we first observe that $S$-units in the sequence $\disunit{\pol}{S}{\alpha} = \{ \pol^n(\alpha)-\pol^m(\alpha) \}_{n > m \geq 0}$ correlate with $S$-integral (with respect to $\alpha$) preperiodic points of corresponding (pre)period.

\begin{lemma} \label{lem:Suint}
Let $K$ be a number field, let $\alpha \in K$, let $\pol \in K[z]$ be a polynomial of degree $d \geq 2$, and let $S$ be a finite set of places of $K$, containing all the archimedean ones and all the places of bad reduction for $\pol$. Then
\begin{align*}
|\relsupreper{\pol}{S}{\alpha}| \geq \big| \big \{ & \beta \in \overline K : \pol^n(\beta)-\pol^m(\beta) = 0 \text{ for some } n > m \geq 0 \\ 
& \text{such that } \pol^n(\alpha) - \pol^m(\alpha) \in \cO_S^* \big \} \big|.
\end{align*}
\end{lemma}

\begin{proof}
Write
$$
\pol(z) = a_0 + a_1 z + \cdots + a_d z^d, \: a_i \in K,
$$
and let $v \notin S$. Then, since $\pol$ has good reduction at $v$, $|a_i|_v \leq 1$ for all $i$, and $|a_d|_v = 1$. Hence, if $|\alpha|_v > 1$, then
$$
|\pol(\alpha)|_v = |a_0 + a_1 \alpha + \cdots + a_d \alpha^d|_v = |\alpha|_v^d > 1,
$$
and so continuing by induction, $|\pol^n(\alpha)|_v = |\alpha|_v^{d^n} > 1$ for all $n \geq 1$. Thus, for all $n > m \geq 0$, $|\pol^n(\alpha) - \pol^m(\alpha)|_v = |\alpha|_v^{d^n} > 1$, and so $\pol^n(\alpha)-\pol^m(\alpha)$ is not an $S$-unit. We hence assume that $|\alpha|_v \leq 1$ for all $v \notin S$. Now, let $n > m \geq 0$, and write
\begin{align*}
\pol^n(z)-\pol^m(z) & = b_0 + b_1 z + \cdots + b_{d^n} z^{d^n} \\ 
& = \pol^n(\alpha)-\pol^m(\alpha) + c_1 (z-\alpha) + \cdots + c_{d^n} (z-\alpha)^{d^n}.
\end{align*}
Note that $b_{d^n} = c_{d^n}$, and for $1 \leq k < d^n$ we have
\begin{equation} \label{eq:Sstuff}
c_{d^n-k} = b_{d^n-k} + \sum_{i = 1}^{k} (-1)^{i-1} \binom{d}{i} \alpha^i c_{d^n-k+i}.
\end{equation}
Let $v \notin S$. Then $|b_i|_v \leq 1$ for all $i$, again using the fact that $\pol$ has good reduction at $v$, and so an easy induction with \eqref{eq:Sstuff} shows that $|c_i|_v \leq 1$ for all $i$. Thus, if $\pol^n(\alpha)-\pol^m(\alpha)$ is an $S$-unit, then for any root $\beta$ of $\pol^n(z)-\pol^m(z)$, we must have $|\beta - \alpha|_v \geq 1$ for all $v \notin S$. Indeed, if $|\beta-\alpha|_v < 1$ for some $v \notin S$, then
$$
0 = |\pol^n(\beta)-\pol^m(\beta)|_v = | \pol^n(\alpha)-\pol^m(\alpha) + c_1 (\beta-\alpha) + \cdots + c_{d^n} (\beta-\alpha)^{d^n} |_v = 1,
$$
a contradiction. That is, such $\beta$ are $S$-integral relative to $\alpha$ (as the relevant $K$-embeddings of $\beta$ are also roots of $\pol^n(z)-\pol^m(z)$), and therefore belong to $\relsupreper{\pol}{S}{\alpha}$, as they are of course preperiodic for $\pol$.
\end{proof}

Hence, to relate $|\relsupreper{\pol}{S}{\alpha}|$ to $|\disunit{\pol}{S}{\alpha} \cap \cO_S^*|$, it will suffice to obtain a lower bound for the number of distinct zeros of $\pol^n-\pol^m$, $n > m \geq 0$. For any complex polynomial $f$, let $\z(f) := \deg \mathrm{rad}(f)$ denote the number of distinct zeros of $f$. Then we have the following.

\begin{lemma} \label{lem:distRootBound}
Let $K$ be a number field and let $\pol \in K[z]$ be a polynomial of degree $d \geq 2$. Then for all $n > m \geq 0$, we have $\z( \pol^n - \pol^m ) \geq \max \{ 1, n-m-2 \}  \max \{ 1, d^{m-1} \} \geq n$.
\end{lemma}

\begin{proof}
First note that by \cite[Theorem~2]{Ba}, $\z(\pol^k-\pol^0) \geq k$ for all $k \geq 1$, since $\pol$ has a cycle of exact period $k$ (with the exception of $k=2$ when $\pol$ is linearly conjugate to $z^2-3/4$, but one can manually check $\z(\pol^2-\pol^0) \geq 2$ in this case). Now, for polynomials $g,h$ with $\z(g) > 1$, \cite[Main~Theorem]{FP} gives
$$
\z(g \circ h) \geq \max \{ 1, \z(g)-2 \} \deg h + 1.
$$
Hence, if $(n,m) \neq (1,0)$ or $\z(\pol-\pol^0) > 1$, the result follows from the above with $g = \pol^{n-m}-\pol^0$, $h = \pol^m$. Otherwise, take $g = \pol^2-\pol$, $h = \pol^{m-1}$.
\end{proof}

\begin{proof}[Proof of Theorem~\ref{thm:SunitBound}]
If $\disunit{\pol}{S}{\alpha} \cap \cO_S^*$ is non-empty, it must contain an $S$-unit of the form $\pol^n(\alpha)-\pol^m(\alpha)$ with
$$
n \geq \sqrt{2 |\disunit{\pol}{S}{\alpha} \cap \cO_S^*| + \frac{1}{4}} - \frac{1}{2}.
$$
Now, from Lemma~\ref{lem:Suint} and Lemma~\ref{lem:distRootBound}, we have
$$
|\relsupreper{\pol}{S}{\alpha}| \geq \z(\pol^n-\pol^m) \geq n \geq \sqrt{2 |\disunit{\pol}{S}{\alpha} \cap \cO_S^*| + \frac{1}{4}} - \frac{1}{2},
$$
and rearranging completes the proof.
\end{proof}

On the other hand, it is worth noting that assuming the $abc$-conjecture, \cite[Corollary~1.6]{GNT15} can be adapted to number fields (see also \cite[Theorem~1.2]{GNT}, which assumes Vojta's conjecture): Given a polynomial $\unipol \in K[z]$ which is not linearly conjugate over $K$ to a polynomial for the form $z^d+c$, and a non-preperiodic point $\alpha$ for $\unipol$, there exists an effectively computable finite set $\cZ$ depending only on $K, \unipol, |S|$ (enlarging $S$ here to contain all places of bad reduction for $\unipol$) and the canonical height $\canheight{\unipol}(\alpha)$ of $\alpha$ with respect to $\unipol$, such that for every $(m,n) \in (\Z_{\geq 0} \times \Z^+) \setminus \cZ$, there is a place $v \notin S$ such that $\alpha$ has \emph{preperiodicity portrait} $(m,n)$ for $\unipol$ modulo $v$. That is, $|\unipol^{m+n}(\alpha)-\unipol^m(\alpha)|_v < 1$, while $|\unipol^{m+k}(\alpha)-\unipol^m(\alpha)|_v, |\unipol^{j+\ell}(\alpha)-\unipol^j(\alpha)|_v \geq 1$ for $1 \leq k < n$, $0 \leq j < m$ and $\ell \geq 1$. Looking to the Newton polygon of $\unipol^{m+n}-\unipol^m$ expanded about $\alpha$, we see this implies that for $(m,n) \notin \cZ$ the \emph{generalised dynatomic polynomial}
$$
\Phi_{\unipol, m,n}(z) := \frac{\Phi_{\unipol,n}(\unipol^m(z))}{\Phi_{\unipol,n}(\unipol^{m-1}(z))}
$$
(where $\Phi_{\unipol,n}$ denotes the \emph{dynatomic polynomial}
$$
\Phi_{\unipol,n}(z) := \prod_{d \mid n} \left( \unipol^d(z)-z \right)^{\mu(n/d)},
$$
and here $\mu$ is the M\"{o}bius function) has a root $\beta$ satisfying $|\alpha-\beta|_v < 1$. Thus, if $\Phi_{\unipol,m,n}$ is irreducible, no root of $\Phi_{\unipol,m,n}$ is $S$-integral relative to $\alpha$. For fixed $\unipol$, the pairs $(m,n)$ such that $\Phi_{\unipol,m,n}$ is reducible appear to be numerically isolated, so perhaps it is the case that $\Phi_{\unipol,m,n}$ is irreducible for all sufficiently large $n$. This would prove Conjecture~\ref{conj:Ih} (assuming $abc$), for non-unicritical polynomials, even without the assumption of $\alpha$ being totally Fatou, but this irreducibility problem seems itself difficult.

\Address


\begin{thebibliography}{30}

\bibitem{Ba} I. N. Baker, \textit{Fixpoints of polynomials and rational functions}, Journal London Math. Soc. \textbf{39} (1964), 615-622.

\bibitem{BIR} M. Baker, S. Ih and R. Rumely, \textit{A finiteness property of torsion points}, Algebra Number Theory \textbf{2} (2008), no. 2, 217-248.

\bibitem{BR1} M. Baker and R. Rumely, \textit{Equidistribution of small points, rational dynamics, and potential theory}, Ann. Inst. Fourier (Grenoble) \textbf{56} (2006(, no. 3, 625-688.

\bibitem{BR} M. Baker and R. Rumely, \textit{Potential theory and dynamics on the Berkovich projective line}, American Mathematical Society, 2010.

\bibitem{B} R. Benedetto, \textit{Dynamics in one non-archimedean variable}, American Mathematical Society, 2019.

\bibitem{BD} F. Berteloot and T.-C. Dinh, \textit{The Mandelbrot set is the shadow of a Julia set}, Discrete and Continuous Dynamical Systems, \textbf{40} (2020), no. 12, 6611-6633.

\bibitem{BG} E. Bombieri and W. Gubler, \textit{Heights in diophantine geometry}, Cambridge University Press, 2006.


\bibitem{Br} H. Brolin, \textit{Invariant sets under iterations of rational functions}, Arkiv f\"{u}r Mathematik, \textbf{6} (1966)

\bibitem{CG} L. Carleson and T. W. Gamelin, \textit{Complex Dynamics}, Springer, New York, 1993.

\bibitem{C} I. Chatzigeorgiou, \textit{Bounds on the Lambert function and their application to the outage analysis of user cooperation}, in IEEE Communications Letters, \textbf{17} (2013), no. 8, 1505-1508.

\bibitem{DH} A. Douady and J. H. Hubbard, \textit{Exploring the Mandelbrot set}, The Orsay notes, \url{http://pi.math.cornell.edu/~hubbard/OrsayEnglish.pdf}.

\bibitem{DKY} L. Demarco, H. Krieger and H. Ye, \textit{Uniform Manin-Mumford for a family of genus 2 curves}, Ann. of Math. (2) \textbf{191} (2020) no. 3, 949-1001.

\bibitem{D} A. Dubickas, \textit{Algebraic numbers with bounded degree and bounded Weil height}, Bull. Aust. MAth. Soc. \textbf{98} (2018), 212-220.

\bibitem{FRL} C. Favre and J. Rivera-Letelier, \textit{Equidistribution quntitative des points de petite hauteur sur la droite projective}, Math. Ann. \textbf{335} (2006), no. 2, 311-361.

\bibitem{F} P. Fili, \textit{A metric of mutual energy and unlikely intersections for dynamical systems}, Preprint, arXiv:1708.08403v1 [math.NT].

\bibitem{FP} C. Fuchs and A Peth\"{o}, \textit{Composite rational functions having a bounded number of zeros and poles}, Proc. Amer. Math. Soc., \textbf{139} (2011), no. 1, 31-38.

\bibitem{GI} D. Grant and S. Ih, \textit{Integral division points on curves}, Composito Math. \textbf{149} (2013), 2011-2035.

\bibitem{GNT15} D. Ghioca, K. Nguyen and T. J. Tucker, \textit{Portraits of preperiodic points for rational maps.} \textit{Math. Proc. Camb. Phil. Soc.} \textbf{159} (2015), 165-186.

\bibitem{GNT} D. Ghioca, K. Nguyen and T. J. Tucker, \textit{Squarefree doubly primitive divisors in dynamical sequences}, Math. Proc. Camb. Phil. Soc. \textbf{164} (2018), no. 3, 551-572.

\bibitem{HS} L.-C. Hsia and J. H. Silverman, \textit{A quantitative estimate for quasi-integral points in orbits}, Pacific J. Math. \textbf{249} (2011), no. 2, 321-342.

\bibitem{IT} S. Ih and T. Tucker, \textit{A finiteness property for preperiodic points of Chebyshev polynomials}, Int. J. Number Theory, \textbf{6} (2010), no. 5, 1011-1025.

\bibitem{I} P. Ingram, \textit{Lower bounds on the canonical height associated to the morphism $\varphi(z) = z^d+c$}, Monatsh. Math. \textbf{157} (2009), no. 1, 69-89.

\bibitem{I2} P. Ingram, \textit{The critical height is a moduli height}, Duke Math. J. \textbf{167} (2018), no. 7, 1311-1346.

\bibitem{K} M. Kosek, \textit{H\"{o}lder exponents of the green function of planar polynomial Julia sets}, Ann. Mat. Pura ed App \textbf{193} (2014), 359-368.

\bibitem{KPS} T. Krick, L. M. Pardo and M. Sombra, \textit{Sharp estimates for the arithmetic Nullstellensatz}, Duke Math. J. \textbf{109} (3), 2001, 521-598.

\bibitem{KLS} H. Krieger, A. Levin, Z. Scherr, T. J. Tucker, Y. Yasufuku and M. E. Zieve, \textit{Uniform boundedness of $S$-units in arithmetic dynamics}, Pacific Journal of Mathematics, \textbf{274} (2015), no. 1, 97-105.

\bibitem{L} M. Lyubich, \textit{Conformal Geometry and Dynamics of Quadratic Polynomials}, Book in preparation, \url{www.math.sunysb.edu/\~mlyubich/book.pdf}.

\bibitem{NC} M. Narvaez-Clauss, \textit{Quantitative equidistribution of Galois orbits of points of small height on the algebraic torus}, Ph.D thesis, Universitat de Barcelona, 2016, available at \url{http://diposit.ub.edu/dspace/bitstream/2445/112532/1/MNC_PhD_THESIS.pdf}.

\bibitem{P} C. Petsche, \textit{S-integral preperiodic points for dynamical systems over number fields}, Bull. London Math. Soc., \textbf{40} (2008), no. 5, 749-758.

\bibitem{Po} C. Pommerenke, \textit{Univalent functions, with a chapter on quadratic differentials by Gerd Jensen}, Studia Math. Lehrb\"{u}cher, \textbf{15} (1975), Vandenhoeck and Ruprecht.

\bibitem{Ra} M. Raynaud, \textit{Sous-vari\'{e}t\'{e}s d'une vari\'{e}t\'{e} ab\'{e}lienne et points de torsion}, `Arithmetic and geometry', Vol. I, 327-252, Progr. Math. 35, Birkh\"{a}user Boston, Boston, MA, 1983.

\bibitem{RL} J. Rivera-Letelier, \textit{Dynamique des fonctions rationnelles sur des corps locaux} (French, with Endlish and French summaries), Ast\'{e}risque \textbf{287} (2003), xv, 147-230. Geometric methods in dynamics. II.

\bibitem{RW} R. Rumely and S. Winburn, \textit{The Lipschitz constant of a non-archimedean rational function}, arXiv:1512.01136 [math.DS]

\bibitem{Si} J. H. Silverman, \textit{The arithmetic of dynamical systems}, Springer, 2007.

\bibitem{Si3} J. H. Silverman, \textit{Integer points, Diophantine approximation, and iteration of rational maps}, Duke Math. J. \textbf{71} (1993), no. 3, 793-829.

\bibitem{Si2} J. H. Silverman, \textit{Moduli Spaces and Arithmetic Dynamics}, volume 30 of \textit{CRM Monograph Series}, AMS, 2012. 

\bibitem{S} C. M. Stroh, \textit{Julia sets of complex polynomials and their implementation on the computer}, Master's thesis, University of Linz, 1997.

\end{thebibliography}
\end{document}